\documentclass[a4paper]{amsart}
\usepackage{amssymb}
\usepackage{hyperref, aliascnt}
\usepackage{graphicx}
\usepackage{xy} \xyoption{all}

\newcommand{\andSep}{\,\,\,\text{ and }\,\,\,}
\newcommand{\CC}{{\mathbb{C}}}
\newcommand{\NN}{{\mathbb{N}}}

\newcommand{\ca}{$C^*$-algebra}
\newcommand{\csuba}{$C^*$-subalgebra}
\newcommand{\subseteqRotatedUp}{\mathrel{\reflectbox{\rotatebox[origin=c]{90}{$\subseteq$}}}}

\DeclareMathOperator{\linspan}{span}
\DeclareMathOperator{\interior}{int_D}
\DeclareMathOperator{\exterior}{cl_D}
\DeclareMathOperator{\Nil}{Nil}

\newenvironment{psmallmatrix}
  {\left(\begin{smallmatrix}}
  {\end{smallmatrix}\right)}

\def\today{\number\day\space\ifcase\month\or   January\or February\or
   March\or April\or May\or June\or   July\or August\or September\or
   October\or November\or December\fi\   \number\year}

\newtheorem{lma}{Lemma}[section]

\newaliascnt{thmCt}{lma}
\newtheorem{thm}[thmCt]{Theorem}
\aliascntresetthe{thmCt}

\newaliascnt{corCt}{lma}
\newtheorem{cor}[corCt]{Corollary}
\aliascntresetthe{corCt}

\newaliascnt{prpCt}{lma}
\newtheorem{prp}[prpCt]{Proposition}
\aliascntresetthe{prpCt}

\theoremstyle{definition}

\newaliascnt{dfnCt}{lma}
\newtheorem{dfn}[dfnCt]{Definition}
\aliascntresetthe{dfnCt}

\newaliascnt{rmkCt}{lma}
\newtheorem{rmk}[rmkCt]{Remark}
\aliascntresetthe{rmkCt}

\newaliascnt{rmksCt}{lma}
\newtheorem{rmks}[rmksCt]{Remarks}
\aliascntresetthe{rmksCt}

\newaliascnt{exaCt}{lma}
\newtheorem{exa}[exaCt]{Example}
\aliascntresetthe{exaCt}

\newaliascnt{pgrCt}{lma}
\newtheorem{pgr}[pgrCt]{}
\aliascntresetthe{pgrCt}

\newcounter{theoremintro}

\newaliascnt{thmIntroCt}{theoremintro}
\newtheorem{thmIntro}[thmIntroCt]{Theorem}
\aliascntresetthe{thmIntroCt}

\newaliascnt{corIntroCt}{theoremintro}
\newtheorem{corIntro}[corIntroCt]{Corollary}
\aliascntresetthe{corIntroCt}

\title{Semiprime ideals in C*-algebras}

\author[Eusebio Gardella]{Eusebio Gardella}
\address{Eusebio Gardella,
Department of Mathematical Sciences, Chalmers University of
Technology and University of Gothenburg, Gothenburg SE-412 96, Sweden.}
\email{gardella@chalmers.se}
\urladdr{www.math.chalmers.se/~gardella}

\author{Kan Kitamura}
\address{Kan Kitamura,
Interdisciplinary Theoretical and Mathematical Sciences Program, RIKEN,
2-1 Hirosawa, Wako Saitama 351-0198, Japan.}
\email{kan.kitamura@riken.jp}

\author{Hannes Thiel}
\address{Hannes~Thiel, 
Department of Mathematical Sciences, Chalmers University of Technology and University of
Gothenburg, Gothenburg SE-412 96, Sweden.}
\email{hannes.thiel@chalmers.se}
\urladdr{www.hannesthiel.org}

\thanks{
The first named author was partially supported by the Swedish Research Council Grant 2021-04561.
The second named author was partially supported by JSPS KAKENHI Grant Number JP22KJ0618 and JST CREST program JPMJCR18T6.
The third named author was partially supported by the Knut and Alice Wallenberg Foundation (KAW 2021.0140).
}

\subjclass[2020]%
{Primary
46L05. % General theory of C*-algebras
Secondary
16N60. % Prime and semiprime associative rings
16W10. % Rings with involution; Lie, Jordan and other nonassociative structures
}
\date{\today}

\begin{document}

%==========================================================================================
\begin{abstract}
We show that a not necessarily closed ideal in a \ca{} is semiprime if and only if it is idempotent, if and only if it is closed under square roots of positive elements.
Among other things, it follows that prime and semiprime ideals in \ca{s} are automatically self-adjoint.

To prove the above, we isolate and study a particular class of ideals, which we call \emph{Dixmier ideals}. 
As it turns out, there is a rich theory of powers and roots for Dixmier ideals. We show that every ideal in a \ca{} is squeezed by Dixmier ideals from inside and outside tightly in a suitable sense, 
from which we are able to deduce information about the ideal in the middle.
\end{abstract}	

\maketitle
%==========================================================================================
%==========================================================================================
\section{Introduction}

%==========================================================================================
Prime and semiprime ideals are fundamental concepts in the structure theory of rings.
They serve as building blocks for factorization and localization techniques which are used to study the global structure of a ring as well as properties of individual elements.
In this paper, we study these ring-theoretic concepts for \ca{s}.
Therefore, by an \emph{ideal} in a \ca{} we will always mean a two-sided ideal that is not necessarily norm-closed. More explicitly, an ideal in a \ca{}~$A$ is an additive subgroup $I\subseteq A$ satisfying $aI\cup Ia\subseteq I$ for all $a\in A$. 
In particular, an ideal in a non-unital \ca{} is not necessarily closed under scalar multiplication. % by complex scalars.

We recall that an ideal $I$ in a \ca{} $A$ is said to be \emph{prime} if $I \neq A$ and if for all ideals $J,K \subseteq A$, whenever $JK$ is contained in $I$ then $J \subseteq I$ or $K \subseteq I$. 
Also, an ideal $I$ is \emph{semiprime} if for all ideals $J \subseteq A$, whenever $J^2$ is contained in $I$ then $J \subseteq I$. 
Equivalently, $I$ is an intersection of prime ideals; see \cite[Section~10]{Lam01FirstCourse2ed}, for example, whose proof still works for non-unital \ca{s}. 

Given ideals $I,J$ in a \ca{}, we follow the ring theoretic convention that~$IJ$ denotes the additive subgroup generated by $\{xy : x \in I, y \in J \}$.
Similarly, $I^n$ denotes the additive subgroup generated by $\{x_1\cdots x_n : x_1,\ldots,x_n \in I \}$.
Note that~$IJ$ and $I^n$ are again ideals.

Our main result is a concise characterization of semiprime ideals in \ca{s}.

%==========================================================================================
\begin{thmIntro}(See \autoref{prp:CharSemiprime}).
\label{MainThm}
Given an ideal $I$ in a \ca{}, the following are equivalent:
\begin{enumerate}
\item
$I$ is semiprime.
\item
$I$ is \emph{idempotent}, that is, $I^2 = I$, which means that for every $x \in I$ there exist $y_j,z_j \in I$ such that $x = y_1z_1 + \ldots + y_nz_n$.
\item
$I^m = I^n$ for some $m \neq n$.
\item
$I$ \emph{factors}, that is, for every $x \in I$ there exist $y,z \in I$ such that $x = yz$.
\item
$I$ is closed under (square or arbitrary) roots of positive elements.
\end{enumerate}
\end{thmIntro}

%==========================================================================================
The equivalence between semiprimeness and idempotency in this context is somewhat surprising since it is easy to produce examples of ideals in rings that are semiprime but not idempotent, or that are idempotent but not semiprime.
It is also notable that (5) is a condition on the positive cone of a \ca{}, which essentially relies on its $\ast$-structure, while the other conditions only refer to the ring structure of the \ca{}.
This suggests that the proof of \autoref{MainThm} is quite intricate, and indeed both the implication `(1)$\Rightarrow$(2)' and its converse ultimately rely on a deep analysis of the shadows of polar decomposition available in \ca{s}.

\medskip

%==========================================================================================
It should be pointed out that \emph{norm-closed} ideals in \ca{s} are well-known to be semiprime (and self-adjoint).
In fact, much more is true:
Every closed ideal in a \ca{} is an intersection of primitive ideals (that is, kernels of irreducible representations on Hilbert spaces), and primitive ideals are closed and prime.
This leads to a natural identification of closed ideals in a \ca{} with closed subsets of its primitive ideal space;
see \cite[Section~II.6.5]{Bla06OpAlgs}.

\emph{Norm-closed} prime ideals in \ca{s} are also well-understood. 
Indeed, for separable \ca{s} they coincide precisely with the primitive ideals - although Weaver showed that this is no longer true for non-separable \ca{s} \cite{Wea03PrimeNotPrim}.
Further, a proper, closed ideal $I$ in a \ca{} $A$ is prime if and only if for all \emph{closed} ideals $J,K \subseteq A$ with $JK \subseteq I$ we have $J \subseteq I$ or $K \subseteq I$;
see, for example, \cite[after Definition~II.5.4.4]{Bla06OpAlgs}.

The theory of not necessarily norm-closed ideals in \ca{s} also has a long history, but it is much less developed and many aspects remain mysterious. 
It is in this area that our results shed new light.
First, it should be noted that many \ca{s} contain a wealth of non-closed prime or semiprime ideals.
Indeed, this is already the case for commutative \ca{s};
see \cite[Chapter~14]{GilJer60RgsCtsFcts}.
In the noncommutative setting, an investigation of non-closed ideals was launched by Pedersen \cite{Ped66MsrThyCAlg1} in the 1960s.
Among other things, he showed that every \ca{} contains a minimal dense ideal, nowadays called the \emph{Pedersen ideal} of the \ca{}, and that this ideal is closed under square roots of positive elements.
It follows from our results that the Pedersen ideal is always semiprime;
see \autoref{exa:PedersenIdeal}.

Non-closed ideals also arise naturally in the study of tracial weights and singular traces.
Given a tracial weight on a \ca{}, the linear span of positive elements with finite weight is an ideal that need not be norm-closed, and this was one of the motivations for Pedersen's investigations \cite{Ped66MsrThyCAlg1, Ped69MsrThyCAlg34}.
Tracial weights on the algebra $\mathcal{B}(\ell^2(\NN))$ of bounded operators on the Hilbert space $\ell^2(\NN)$ that vanish on finite rank operators are called singular traces, and these play a crucial role in noncommutative geometry \cite{LorSukZan13SingTrace}.
The wealth of singular traces is closely connected to the plethora of non-closed ideals in $\mathcal{B}(\ell^2(\NN))$, including for example the Schatten $p$-ideals.
A description of the numerous non-closed prime and semiprime ideals in $\mathcal{B}(\ell^2(\NN))$ was given by Morris--Salinas \cite{MorSal73SemiprimeIdlsBH}.
This should be contrasted with the fact that $\mathcal{B}(\ell^2(\NN))$ contains only three closed ideals: the zero ideal, the compact operators, and $\mathcal{B}(\ell^2(\NN))$.

\medskip

%==========================================================================================
\autoref{MainThm} has several interesting consequences. 
We list some of them here:

It is well-known that norm-closed ideals in \ca{s} are self-adjoint.
We show that the same holds for arbitrary semiprime ideals, while it fails for general ideals.
However, we do not know an easy proof for this simply stated result due to the difficulty in extracting $*$-algebraic information from purely ring theoretic assumption on ideals.
In fact, self-adjointness will be a consequence of a much stronger result, for which we recall some terminology.
We say that an ideal $I$ in a \ca{} $A$ is \emph{positively spanned} if it is the linear span of its positive part $I_+ := I \cap A_+$, and we say that $I$ is \emph{hereditary} if whenever $a\in A$ and $x\in I$ satisfy $0\leq a\leq x$, then $a\in I$.

%==========================================================================================
\begin{corIntro} (See \autoref{prp:SemiprimeDixmier}).
\label{corB}
Semiprime ideals in \ca{s} are positively spanned (hence self-adjoint and closed under scalar multiplication) and hereditary.
\end{corIntro}

%==========================================================================================
Note that \autoref{corB} applies to prime ideals and in particular maximal ideals; 
see \autoref{prp:MaximalIdeal}. 

It follows from Cohen's factorization theorem that norm-closed ideals in a \ca{} factor, and in particular are idempotent.
While this is no longer true for arbitrary ideals, we obtain the following interesting dichotomy:

%==========================================================================================
\begin{corIntro} (See \autoref{prp:Dichotomy}).
For an ideal $I$ in a \ca{}, we have either 
\[
I = I^2 = I^3 = \ldots \ 
\]
or 
\[
I \supsetneqq I^2 \supsetneqq I^3 \supsetneqq \ldots. 
\]
\end{corIntro}

%==========================================================================================
Recall that a module $E$ over a \ca\ $A$ is said to be \emph{unital}
if $EA=E$.
Another consequence of \autoref{MainThm} is that trace ideals of projective, unital modules over \ca{s} are semiprime;
see \autoref{prp:TraceIdeals}.
This is of interest for studying the `ring-theoretic Cuntz semigroup' of a \ca{} \cite{AntAraBosPerVil23arX:CuRg, AntAraBosPerVil24arX:IdlsCuRg}.

While \autoref{MainThm} characterizes when an ideal in a \ca{} is semiprime, our methods to prove the theorem allow us to relate an arbitrary ideal with its \emph{semiprime closure}, the smallest semiprime ideal containing it.
Further terminology will be explained before \autoref{prp:SemiprimeClosure}.

%==========================================================================================
\begin{corIntro} (See \autoref{prp:SemiprimeClosure} and \autoref{prp:NilRadicals}).
The semiprime closure of an ideal $I$ in a \ca{} $A$ is the union of some increasing sequence of ideals that are radical over $I$.
Further, every one-sided nil ideal in $A/I$ is contained in the prime radical of $A/I$. 
\end{corIntro}

%==========================================================================================
In particular, \autoref{prp:NilRadicals} verifies K\"{o}the's conjecture for rings that arise as the quotient of a \ca{} by a (not necessarily norm-closed) ideal;
see \autoref{rmk:KoetheConjecture}.

\medskip

%==========================================================================================
%We mention another consequence of \autoref{MainThm} and the techniques we developed for its proof. 
The ideals of a \ca{} $A$ form a complete lattice with respect to inclusion, and the infimum and supremum of a family of ideals is given by their intersection and sum, respectively. 
Norm-closed ideals of $A$ form a sublattice, but it is typically not a complete sublattice, because infinite sums of norm-closed ideals need not be norm-closed. 
While many consequences of \autoref{MainThm} suggest similarities for semiprime ideals and norm-closed ones, the family of semiprime ideals behaves better at this point: 

%==========================================================================================
\begin{thmIntro} (See \autoref{prp:LatticeSemiprimeIdeals}).
The semiprime ideals of a \ca{} form a complete sublattice of the lattice of all ideals. 
\end{thmIntro}

\medskip

%==========================================================================================
To obtain the main theorem, we consider a particular class of ideals in \ca{s} with a well-behaved notion of powers and roots. 
We shall call this class Dixmier ideals. 
More precisely, by a \emph{Dixmier ideal} in a \ca{}~$A$ we mean a positively spanned, hereditary ideal $I$ such that its positive part $I_+$ is a \emph{strongly invariant}, that is, $x^*x \in I_+$ if and only if $xx^*\in I_+$ for all $x \in A$;
see \autoref{dfn:DixmierIdeal}.
We will justify our terminology right after \autoref{dfn:DixmierIdeal}.

Given a Dixmier ideal $I$ and $s\in(0,\infty)$, we show that there exists a unique Dixmier ideal $I^s$ whose positive part is $\{a^s:a\in I_+\}$;
see \autoref{dfn:PowerDixmierIdeal}. It is immediate that $I^1=I$, 
and we show in \autoref{prp:PowersDixmierIdeal} that $I^{s+t} = I^sI^t$ and $(I^s)^t = I^{st}$ for all $s,t \in (0,\infty)$. 
In particular, for $s=2$ our definition agrees with the usual algebraic definition of $I^2$ as the additive subgroup generated by $\{xy : x,y \in I\}$.
Moreover, the ideal $I^{\frac{1}{n}}$ is the canonical $n$-th root of a Dixmier ideal $I$, since it is the unique Dixmier ideal $J$ satisfying $J^n=I$; 
see \autoref{prp:UniqueDixmierRoot}.

While not every ideal in a \ca{} is Dixmier, every ideal is closely embraced by Dixmier ideals.
More precisely, given any ideal $I$, there exist Dixmier ideals $J$ and $K$ such that $J \subseteq I \subseteq K$, and such that for arbitrarily small $\varepsilon\in(0,1)$ we have $K^{1+\varepsilon} \subseteq J$ and $K \subseteq J^{1-\varepsilon}$;
see \autoref{prp:IntExt}.

\medskip

%==========================================================================================
The plan of this paper is as follows.
In \autoref{sec:RelationPositive} we lay the technical foundation by investigating several comparison relations for positive elements in a \ca{}.
Next, we introduce the notion of Dixmier ideals and develop the theory of their powers and roots in \autoref{sec:DixmierIdeals}. 
We also consider the complete lattice of Dixmier ideals and show that every ideal is tightly embraced by Dixmier ideals in \autoref{sec:Lattice}. 
In \autoref{sec:SemiprimeIdeals}, we finally focus on semiprime ideals and exploit the theory of Dixmier ideals to obtain \autoref{MainThm} and other results. 

%==========================================================================================
\subsection*{Conventions}

%==========================================================================================
%We note general notation we use throughout the paper. 
For a not necessarily unital \ca\ $A$, we let~$A_+$ denote the set of positive elements in $A$.
For $x \in A$, we put $|x|:=(x^*x)^{\frac{1}{2}}$. 

We let $\NN$ denote the set of natural numbers including $0$.

%==========================================================================================
\subsection*{Acknowledgements}

%==========================================================================================
Part of this work was completed while the first and third named authors were attending the Thematic Program on Operator Algebras and Applications at the Fields Institute for Research in Mathematical Sciences in August and September 2023, and they gratefully acknowledge the support and hospitality of the Fields Institute.

The authors thank Jan {\v{Z}}emli{\v{c}}ka for valuable comments on the radical theory of rings.
The second named author thanks Yasuyuki Kawahigashi, who was his advisor, for support and encouragement throughout his PhD course.
The third named author thanks George Elliott, Francesc Perera, Chris Schafhauser, Eduard Vilalta, and Stuart White for interesting discussions on ideals in \ca{s}.

The authors thank the anonymous referees for their valuable feedback on the submitted version of this paper.

%==========================================================================================
%==========================================================================================
\section{Comparison relations for positive elements}
\label{sec:RelationPositive}

%==========================================================================================
In this section, we introduce several binary relations for positive elements in a \ca{}; 
see \autoref{dfn:Relations}. 
We then lay the technical foundation to the subsequent sections by investigating how these relations, especially the relation $\lhd$, behave under sums and powers; 
see \autoref{prp:PropertiesRelations} and \autoref{prp:PowerLhd}. 
The proofs rely on the versions of polar decomposition and Riesz properties that are available in \ca{s}; 
see \autoref{prp:PolarRiesz}.

\medskip

%==========================================================================================
If $M$ is a von Neumann algebra, then every $x \in M$ admits a polar decomposition $x = v|x|$ with a partial isometry $v \in M$, and such that $v^*x=|x|$. 
In particular, $v^*xv^*=|x|v^*=x^*$.
More generally, given $a \in M_+$ with $x^*x \leq a$, there exists a contractive element $w \in M$ such that $x = wa^{\frac{1}{2}}$;
see, for example, \cite[Sections~I.5.2, I.9.2]{Bla06OpAlgs}.
Using that the bidual of a \ca{} is a von Neumann algebra, we obtain shadows of these results for elements in a \ca, which both are often referred to as `polar decomposition' in \ca{s}.
We will also often use Pedersen's version of Riesz Decomposition and Riesz Refinement for positive elements in \ca{s}.
We recall the precise statements for the convenience of the reader.

%==========================================================================================
\begin{prp}
\label{prp:PolarRiesz}
Let $A$ be a \ca{}.
\begin{enumerate}
\item
(Polar Decomposition I)
Let $x \in A$.
Then there exists a partial isometry $v \in A^{**}$ such that $x = v|x|$ in $A^{**}$ and $vy \in A$ for every $y \in \overline{|x|A}$. 
In particular, for $\varepsilon \in (0,1)$ and $z := v|x|^\varepsilon \in A \subseteq A^{**}$, we have
\[
x = z|x|^{1-\varepsilon}. % \big( v|x|^\varepsilon \big) |x|^{1-\varepsilon} \text{ with } v|x|^\varepsilon \in A.
\]

\item
(Polar decomposition II)
Let $x \in A$ and $a \in A_+$ satisfying $x^*x \leq a$.
Then there exists a contraction $w \in A^{**}$ such that $x = wa^{\frac{1}{2}}$ and we have $wy \in A$ for every $y \in \overline{a^{\frac{1}{2}}A} = \overline{aA}$. 
In particular, for $\varepsilon \in (0,\frac{1}{2})$ and $z := wa^\varepsilon$, we have
\[
x = z a^{\frac{1}{2}-\varepsilon}. %\big( wa^\varepsilon \big) a^{\frac{1}{2}-\varepsilon} \text{ with } wa^\varepsilon \in A.
\]

\item
(Riesz Refinement)
Let $x_j,y_k \in A$ % $x_1,\ldots,x_m, y_1,\ldots,y_n \in A$ 
satisfy $\sum_{j=1}^m x_j x_j^* = \sum_{k=1}^n y_k^* y_k$.
Then there exist $z_{jk} \in A$ for $j=1,\ldots,m$ and $k=1,\ldots,n$ such that
\[
x_j^* x_j = \sum_{k=1}^n z_{jk}^*z_{jk} \text{ for each $j$}, \andSep 
\sum_{j=1}^m z_{jk}z_{jk}^* = y_k y_k^* \text{ for each $k$}.
\]
\item
(Riesz Decomposition)
Let $x_j,y_k \in A$ % $x_1,\ldots,x_m, y_1,\ldots,y_n \in A$ 
satisfy $\sum_{j=1}^m x_j x_j^* \leq \sum_{k=1}^n y_k^* y_k$.
Then there exist $z_{jk} \in A$ for $j=1,\ldots,m$ and $k=1,\ldots,n$ such that
\[
x_j^* x_j = \sum_{k=1}^n z_{jk}^*z_{jk} \text{ for each $j$}, \andSep
\sum_{j=1}^m z_{jk}z_{jk}^* \leq y_k y_k^* \text{ for each $k$}.
\]
\end{enumerate}
\end{prp}
\begin{proof}
For~(1), consider the polar decomposition $x = v|x|$ of $x$ in $A^{**}$.
Given $y \in \overline{|x|A}$, we have
\[
vy
\in v\overline{|x|A}
\subseteq \overline{v|x|A}
= \overline{xA}
\subseteq A.
\]
Statement~(2) is shown similarly, using \cite[Sections~I.5.2, I.9.2]{Bla06OpAlgs}.
See also \cite[Proposition~II.3.2.1]{Bla06OpAlgs} and \cite[Proposition~1.4.5]{Ped79CAlgsAutGp}.

For~(3) and~(4), see \cite[Proposition~1.1]{Ped69MsrThyCAlg34}. 
\end{proof}

%==========================================================================================
\begin{dfn}
\label{dfn:Relations}
Let $A$ be a \ca.
We define binary relations $\approx$, $\lessapprox$, $\lessapprox_+$ and~$\lhd$ on $A_+$ as follows. Given $a,b \in A_+$:
\begin{itemize}
\item
We set $a \approx b$ if there exists $x \in A$ such that $a = x^*x$ and $b = xx^*$.
\item
We set $a \lessapprox b$ if there exists $a' \in A_+$ such that $a \approx a' \leq b$.
\item
We set $a \lessapprox_+ b$ if there exist $n$ and $a_1,\ldots,a_n, b_1,\ldots,b_n \in A_+$ such that
\[
a = a_1 + \ldots + a_n, \quad
a_1 \approx b_1, \ \ldots \ ,
a_n \approx b_n, \andSep
b_1 + \ldots + b_n \leq b.
\]
\item
We set $a \lhd b$ if there exist $m \in \NN$ such that $a \lessapprox_+ mb$.
\end{itemize}
\end{dfn}

%==========================================================================================
The main technical result of this section is that positive elements $a$ and $b$ ia a \ca{} satisfy $a \lhd b$ if and only if $a^s \lhd b^s$ for some (equivalently, every) $s \in (0,\infty)$;
see \autoref{prp:PowerLhd}.
We build up to this result by proving that the relations $\approx$, $\lessapprox$ and $\lessapprox_+$ have a similar invariance when passing to powers.
First, we establish some basic properties of the relations from \autoref{dfn:Relations}.

%==========================================================================================
\begin{lma}
\label{prp:SwitchLeqApprox}
Let $A$ be a \ca, and let $a,b,c \in A_+$ satisfy $a \leq b \approx c$.
Then there exists $d \in A_+$ such that $a \approx d \leq c$.
\end{lma}
\begin{proof}
Choose $x \in A$ such that $b = x^*x$, and $xx^* = c$.
Applying Riesz Decomposition for $a = a^{\frac{1}{2}}a^{\frac{1}{2}} \leq x^*x$ (\autoref{prp:PolarRiesz}~(4) with $m=n=1$), we obtain $z \in A$ such that
\[
a = a^{\frac{1}{2}}a^{\frac{1}{2}} = z^*z, \andSep
zz^* \leq xx^* = c,
\]
which shows that $a \approx zz^* \leq c$, as desired.
\end{proof}

%==========================================================================================
\begin{lma}
\label{prp:PropertiesRelations}
Let $A$ be a \ca.
Then:
\begin{enumerate}
\item
The relation $\approx$ on $A_+$ is an equivalence relation.
\item
The relation $\lessapprox$ is a pre-order (reflexive and transitive).
\item
The relation $\lessapprox_+$ is a pre-order that is additive in the sense that $a_1 \lessapprox_+ b_1$ and $a_2 \lessapprox_+ b_2$ imply $a_1+a_2 \lessapprox_+ b_1+b_2$.
\item
The relation $\lhd$ is an additive pre-order. % satisfying $2a \lhd a$ for all $a \in A_+$.
\item
The relations $\lessapprox$, $\lessapprox_+$ and $\lhd$ satisfy Riesz decomposition:
If $a,b_1,b_2 \in A_+$ satisfy $a \lessapprox b_1+b_2$, then there exist $a_1,a_2 \in A_+$ such that
\[
a = a_1+a_2, \quad
a_1 \lessapprox b_1, \andSep
a_2 \lessapprox b_2,
\]
and analogously with $\lessapprox_+$ or $\lhd$ in place of $\lessapprox$.
\end{enumerate}
\end{lma}
\begin{proof}
It is clear that $\approx$, $\lessapprox$ and $\lessapprox_+$ are reflexive, and that $\approx$ is symmetric.
To show that $\approx$ is transitive, let $a,b,c \in A_+$ satisfy $a \approx b \approx c$.
Choose $x,y \in A$ such that
\[
a = x^*x, \quad
xx^* = b = y^*y, \andSep
yy^* = c.
\]
Using Riesz Refinement for $xx^* = y^*y$ (\autoref{prp:PolarRiesz}~(3) with $m=n=1$), there exists $z \in A$ such that
\[
x^*x = z^*z, \andSep
zz^* = yy^*,
\]
which shows that $a \approx c$.
Thus, $\approx$ is transitive.
(See also \cite[Theorem~3.5]{Ped98FactorizationCAlg}.)

We check that $\lessapprox$ is transitive.
Let $a,b,c\in A_+$ satisfy $a\lessapprox b\lessapprox c$. 
Find $a',b'\in A_+$ such that $a\approx a'\leq b\approx b'\leq c$. 
By \autoref{prp:SwitchLeqApprox}, there exists $d\in A_+$ such that $a'\approx d\leq b'$.
Using transitivity of $\approx$, we get $a\approx d\leq c$, as desired.

It is clear that $\lessapprox_+$ is additive.
Further, it follows from Riesz Decomposition (\autoref{prp:PolarRiesz}~(4)) that $\lessapprox_+$ is transitive.
Using that $\lessapprox_+$ is an additive partial order, the same follows for $\lhd$.
This verifies~(1)--(4).
	
We show Riesz Decomposition for $\lessapprox_+$. 
Let $a,b_1,b_2 \in A_+$ satisfy $a \lessapprox_+ b_1+b_2$.
By definition, we obtain $x_j \in A$ such that
\[
a = \sum_{j=1}^m x_j^*x_j, \andSep
\sum_{j=1}^m x_jx_j^* \leq b_1+b_2.
\]
Applying Riesz Decomposition (\autoref{prp:PolarRiesz}~(4)), we obtain $z_{jk} \in A$ for $j=1,\ldots,m$ and $k=1,2$ such that
\[
x_j^*x_j = \sum_{k=1}^2 z_{jk}^*z_{jk} \quad \text{for $j=1,\ldots,m$}, \andSep
\sum_{j=1}^m z_{jk}z_{jk}^* \leq b_k \quad \text{for $k=1,2$}.
\]
Set
\[
a_k := \sum_{j=1}^m z_{jk}^*z_{jk} \quad \text{ for $k=1,2$}.
\]
Then
\[
a 
= \sum_{j=1}^m x_j^*x_j
= \sum_{j=1}^m \sum_{k=1}^2 z_{jk}^*z_{jk}
= \sum_{k=1}^2 \sum_{j=1}^m  z_{jk}^*z_{jk}
= a_1 + a_2,
\]
and $a_k \lessapprox_+ b_k$ for $k=1,2$, as desired. 
Riesz Decomposition for $\lessapprox$ follows by letting $m=1$ in this argument. 
	
To verify Riesz Decomposition for $\lhd$, let $a,b_1,b_2 \in A_+$ satisfy $a \lhd b_1+b_2$.
Choose $m\in\NN$ such that $a \lessapprox_+ m(b_1+b_2) = (mb_1) + (mb_2)$.
Using that $\lessapprox_+$ satisfies Riesz Decomposition, we obtain $a_1,a_2 \in A_+$ such that $a=a_1+a_2$ and $a_k \lessapprox_+ mb_k$, and so $a_k \lhd b_k$ for $k=1,2$, as desired.
\end{proof}

%==========================================================================================
We will now study to what extent the relations from \autoref{dfn:Relations} are preserved under taking powers and roots.

%==========================================================================================
\begin{lma}
\label{prp:PowerApprox}
Let $A$ be a \ca, let $a,b \in A_+$, and let $s \in (0,\infty)$.
Then $a \approx b$ if and only if $a^s \approx b^s$.
\end{lma}
\begin{proof}
First assume that $a \approx b$.
Choose $x \in A$ such that $a = x^*x$ and $b = xx^*$.
By Polar Decomposition (\autoref{prp:PolarRiesz} (1)), there exists a partial isometry $v \in A^{**}$ such that $x = v|x|$ in $A^{**}$, and such that $y := v|x|^s$ belongs to~$A$.
Then we have
\[
y^*y = |x|^s v^*v |x|^s = |x|^{2s} = a^s, \andSep
yy^* = v|x|^{2s}v^* = v(x^*x)^sv^* = (xx^*)^s = b^s.
\]
This shows that $a^s \approx b^s$.
	
Conversely, assuming that $a^s \approx b^s$, we can apply the above argument for $\frac{1}{s}$, $a^s$ and $b^s$ and deduce that
\[
a = (a^s)^{\frac{1}{s}} \approx (b^s)^{\frac{1}{s}} = b,
\]
as desired.
\end{proof}

%==========================================================================================
\begin{lma}
\label{prp:PowerLessapprox}
Let $A$ be a \ca.
Then $a \lessapprox b$ if and only if $a^s \lessapprox b^s$, for all $a,b \in A_+$ and $s \in (0,\infty)$.
\end{lma}
\begin{proof}
As in the proof of \autoref{prp:PowerApprox}, it suffices to show that for every $s \in (0,\infty)$ the following statement holds:
\[
(C_s) : \text{For all $a,b \in A_+$, if $a \lessapprox b$ then $a^s \lessapprox b^s$.}
\]

\textbf{Step~1:} 
\emph{Given $s \in (0,1]$, we verify $(C_s)$.}
Let $a,b \in A_+$ satisfy $a \lessapprox b$.
Choose $c \in A_+$ such that $a \approx c \leq b$.
Using \autoref{prp:PowerApprox} at the first step, and using that the function $t \mapsto t^s$ is operator monotone, \cite[Proposition~II.3.1.10]{Bla06OpAlgs}, at the second step, we get
\[
a^s \approx c^s \leq b^s.
\]
and thus $a^s \lessapprox b^s$, as desired.

\medskip
	
\textbf{Step~2:} 
\emph{We verify $(C_2)$.}
Let $a,b \in A_+$ satisfy $a \lessapprox b$.
Choose $c \in A_+$ such that $a \approx c \leq b$.
By \autoref{prp:PowerApprox}, we have $a^2 \approx c^2$.
Further, we get
\begin{align*}
c^2 
&= c^{\frac{1}{2}}cc^{\frac{1}{2}}
\leq c^{\frac{1}{2}}bc^{\frac{1}{2}}
= \big( c^{\frac{1}{2}} b^{\frac{1}{2}} \big) \big( b^{\frac{1}{2}}c^{\frac{1}{2}} \big) \\
&\approx \big( b^{\frac{1}{2}}c^{\frac{1}{2}} \big) \big( c^{\frac{1}{2}} b^{\frac{1}{2}} \big)
=  b^{\frac{1}{2}} c b^{\frac{1}{2}} 
\leq b^{\frac{1}{2}} b b^{\frac{1}{2}} 
= b^2.
\end{align*}
Using that $\lessapprox$ is transitive by \autoref{prp:PropertiesRelations}, we deduce that $a^2 \lessapprox b^2$.

\medskip 

\textbf{Step~3:} 
\emph{Given $r \in (0,1]$, we verify $(C_s)$ for all $s=2^nr$ by induction on $n\in\NN$.}
The case $n=0$ was shown in Step~1.
Assume that we have verified $(C_{2^nr})$. 
To verify $(C_{2^{n+1}r})$, let $a,b \in A_+$ satisfy $a \lessapprox b$.
By Step~2, we have $a^2 \lessapprox b^2$.
Using that $(C_{2^{n}r})$ holds by assumption, we get
\[
a^{2^{n+1}r} 
= \big( a^2 \big)^{2^nr}
\lessapprox \big( b^2 \big)^{2^nr}
= b^{2^{n+1}r}.
\]
	
This finishes the proof, since every $s \in (0,\infty)$ is of the form $s=2^nr$ for suitable $n \in \NN$ and $r\in(0,1]$.
\end{proof}

%==========================================================================================
\begin{lma}
\label{prp:PowerSumLhd}
Let $A$ be a \ca, let $a_1,\ldots,a_n \in A_+$, and let $s \in (0,\infty)$.
Then 
\[
\sum_{j=1}^{n} a_j^s \lhd \Bigl(\sum_{j=1}^{n} a_j\Bigr)^s, \andSep
\Bigl(\sum_{j=1}^{n} a_j\Bigr)^s \lhd \sum_{j=1}^{n} a_j^s.
\]
\end{lma}
\begin{proof}
First, we verify that $\sum_{j=1}^{n} a_j^s \lhd \bigl(\sum_{j=1}^{n} a_j\bigr)^s$.
We have $a_j \leq \sum_{j=1}^{n} a_j$, and therefore $a_j^s \lessapprox \bigl(\sum_{j=1}^{n} a_j\bigr)^s$ by \autoref{prp:PowerLessapprox}.
Thus $\sum_{j=1}^{n} a_j^s \lessapprox_+ n\bigl(\sum_{j=1}^{n} a_j\bigr)^s$, which gives the desired result.
	
\medskip
	
Next, we verify the second part of the statement for $s$ in the range $[1,\infty)$.
We have $a_j^s \leq \sum_{j=1}^{n} a_j^s$.
Using that the function $t \mapsto t^{\frac{1}{s}}$ is operator monotone, we get
\[
\sum_{j=1}^{n} a_j = \sum_{j=1}^{n} (a_j^s)^{\frac{1}{s}} \leq \sum_{j=1}^{n}\Big( \sum_{k=1}^{n} a_k^s \Big)^{\frac{1}{s}} = n\Big( \sum_{j=1}^{n} a_j^s \Big)^{\frac{1}{s}}.
\]
Applying \autoref{prp:PowerLessapprox}, we get
\[
\Bigl(\sum_{j=1}^{n} a_j \Bigr)^{s} 
\lessapprox 
\Bigl(n\Big( \sum_{j=1}^{n} a_j^s \Big)^{\frac{1}{s}}\Bigr)^s
= n^s \sum_{j=1}^{n} a_j^s ,
\]
which implies that $\bigl(\sum_{j=1}^{n} a_j\bigr)^s \lhd \sum_{j=1}^{n} a_j^s$.

\medskip

Finally, we prove the second inequality for $s \in (0,1)$.
%Finally, we verify the second part of the statement for $s$ in the range $(0,1)$.
We first establish:

\textbf{Claim:}
\emph{Given $e,f \in A$ and $d \in A_+$ with $ed^{1+n}=fd^{1+n}$, we have $ed=fd$.}
It suffices to prove the claim for $n=1$, since then repeated application reduces the exponent of $d$ to $1$.
Thus, assume that $ed^2 = fd^2$. 
Then $(e-f)d^2(e-f)^*=0$, and the $C^*$-identity gives $(e-f)d=0$.
This proves the claim.

\medskip

Set $a := \sum_{j=1}^n a_j$. 
For each $j$, using that $a_j \leq a$, we apply \cite[Proposition~I.5.2.4]{Bla06OpAlgs} in $A^{**}$ to obtain a contractive element $v_j \in A^{**}$ such that
\[
a_j^{\frac{1}{2}} = v_j a^{\frac{1}{2}}.
\]
Then
\[
a 
= \sum_{j=1}^n a_j
= \sum_{j=1}^n a^{\frac{1}{2}}v_j^*v_j a^{\frac{1}{2}}
= a^{\frac{1}{2}} \Big( \sum_{j=1}^n v_j^*v_j \Big) a^{\frac{1}{2}}.
\]
Choose $n \geq 1$ such that $\frac{1}{n+1} \leq s$ and set $t := \frac{1-s}{n}$ and $r := 1-(n+1)t$.
Then $r \geq 0$, $1 = r + (n+1)t$ and $nt = 1-s$.
We have
\[
a^{\frac{nt}{2}} a^{\frac{t}{2}} a^{r} a^{\frac{t}{2}} a^{\frac{nt}{2}}
= a
= a^{\frac{1}{2}} \Big( \sum_{j=1}^n v_j^*v_j \Big) a^{\frac{1}{2}}
= a^{\frac{nt}{2}}a^{\frac{t}{2}}a^{\frac{r}{2}} \Big( \sum_{j=1}^n v_j^*v_j \Big) a^{\frac{r}{2}}a^{\frac{t}{2}}a^{\frac{nt}{2}}.
\]
Apply the Claim to `cancel' $a^{\frac{nt}{2}}$ on the right (and analogously on the left), to get
\[
a^s
= a^{\frac{t}{2}} a^{r} a^{\frac{t}{2}} 
= a^{\frac{t}{2}} a^{\frac{r}{2}} \Big( \sum_{j=1}^n v_j^*v_j \Big) a^{\frac{r}{2}} a^{\frac{t}{2}}
= a^{\frac{s}{2}} \Big( \sum_{j=1}^n v_j^*v_j \Big) a^{\frac{s}{2}}.
\]
Using that $v_j a^{\frac{1}{2}} \in A$, it follows that $x_j := v_j a^{\frac{s}{2}}$ belongs to $A$ as well.
Using also that the map $t \mapsto t^s$ is operator concave (\cite[Theorem~V.2.5]{Bha97Matrix}), and that therefore $wd^sw^* \leq (wdw^*)^s$ for every contraction $w$ and every positive element $d$ (\cite[Theorem~V.2.3(ii)]{Bha97Matrix}), we get
\[
\Bigl(\sum_{j=1}^{n} a_j \Bigr)^{s} 
= a^s	
= \sum_{j=1}^n x_j^*x_j
\lessapprox_+ \sum_{j=1}^n x_jx_j^*
= \sum_{j=1}^n v_ja^sv_j^*
\leq \sum_{j=1}^n \big( v_jav_j^* \big)^s
= \sum_{j=1}^n a_j^s,
\]
as desired.
\end{proof}

%==========================================================================================
\begin{prp}
\label{prp:PowerLhd}
Let $A$ be a \ca, let $a,b \in A_+$, and let $s \in (0,\infty)$.
Then $a \lhd b$ if and only if $a^s \lhd b^s$.
\end{prp}
\begin{proof}
As in the proof of \autoref{prp:PowerApprox}, it suffices to show the forward implication.
Let $a,b \in A_+$ satisfy $a \lhd b$, and let $s \in (0,\infty)$.
By definition, we obtain natural numbers $n,m\in\NN$ with $n \geq 1$, and elements $a_1,\ldots,a_n, b_1,\ldots,b_n \in A_+$ such that
\[
a = \sum_{j=1}^n a_j, \quad
a_j \approx b_j \quad \text{for $j=1,\ldots,n$}, \andSep
\sum_{j=1}^nb_j \leq mb.
\]
	
By \autoref{prp:PowerApprox}, we have $a_j^s \approx b_j^s$ for each $j$, and consequently $\sum_{j} a_j^s \lhd \sum_{j} b_j^s$.
Applying \autoref{prp:PowerSumLhd} at the second and fourth step, and using also \autoref{prp:PowerLessapprox} at the fifth step, we get
\[
a^s
= \Big( \sum_{j=1}^n a_j \Big)^s
\lhd \sum_{j=1}^n a_j^s
\lhd \sum_{j=1}^n b_j^s
\lhd \Big( \sum_{j=1}^n b_j \Big)^s
\lhd ( mb )^s
\lhd b^s.
\]
Since $\lhd$ is transitive by \autoref{prp:PropertiesRelations}, we deduce that $a^s \lhd b^s$.
\end{proof}

%==========================================================================================
%==========================================================================================
\section{Powers and roots of Dixmier ideals}
\label{sec:DixmierIdeals}

%==========================================================================================
In this section, we introduce and develop the concept of Dixmier ideals in \ca{s}; 
see \autoref{dfn:DixmierIdeal}.
Every norm-closed ideal is Dixmier, and every Dixmier ideal is self-adjoint and closed under scalar multiplication.

We show that Dixmier ideals in a \ca{} naturally correspond to strongly invariant order ideals of its positive cone;
see \autoref{prp:DixmierVsStrInvOrderIdeal}.
Using that the relation $\lhd$ from \autoref{sec:RelationPositive} determines strongly invariant order ideals (\autoref{prp:LHD}), we apply \autoref{prp:PowerLhd} to show that there is a well-behaved theory of powers and roots for strongly invariant order ideals, and consequently for Dixmier ideals;
see \autoref{prp:PowerOrderIdeal} and \autoref{prp:PowersDixmierIdeal}. 

%==========================================================================================
\begin{pgr}
\label{pgr:OrderIdeal}
Given a \ca{} $A$, an \emph{order ideal} in $A_+$ is an additive submonoid $M \subseteq A_+$ that is hereditary in the sense that if $a,b \in A_+$ satisfy $a \leq b$, and $b$ belongs to $M$, then so does $a$.
Note that order ideals in $A_+$ are also closed under multiplication by positive real scalars.
For further details, we refer to \cite{Eff63OrderIdlCAlg}.

Following Pedersen~\cite{Ped66MsrThyCAlg1}, we say that an order ideal $M \subseteq A_+$ is \emph{invariant} if $x^*Mx \subseteq M$ for all $x\in A$, or equivalently $u^*Mu=M$ for all unitaries~$u$ in the minimal unitization $\widetilde{A}$.
Further $M$ is \emph{strongly invariant} if for all $x \in A$, we have $x^*x \in M$ if and only if $xx^* \in M$.
If $M$ is strongly invariant, then it is invariant, but the converse does not hold;
see \cite[Proposition~1.3]{Ped69MsrThyCAlg34}.

We say that an ideal~$I$ in $A$ is \emph{positively spanned} if $I = \linspan(I_+)$ for $I_+ := A_+ \cap I$.
We say $I$ is \emph{self-adjoint} if $I=I^*$, namely if for every $a \in A$ we have $a \in I$ if and only if $a^* \in I$.
We further say that $I$ is \emph{closed under scalar multiplication} if $\lambda x \in I$ for all $\lambda\in\CC$ and $x \in I$.
Positively spanned ideals are clearly self-adjoint and closed under scalar multiplication.
As noted before, ideals in \ca{s} are not necessarily self-adjoint nor closed under scalar multiplication;
see \cite[Examples~II.5.2.1]{Bla06OpAlgs}.

An ideal $I$ is \emph{hereditary} if $I_+$ is an order ideal.
In this case, $I_+$ is an invariant order ideal.
As shown in \cite[Corollary~1.2]{Ped66MsrThyCAlg1}, the map $I \mapsto I_+$ establishes a natural bijection between positively spanned, hereditary ideals of $A$ (which are called order-related ideals in \cite{Ped66MsrThyCAlg1}) and invariant order ideals of $A_+$, with inverse given by sending an invariant order ideal $M$ to $\linspan(M)$.

In parts of the literature, `positively spanned' ideals are referred to as `positively generated'.
We avoid the latter, since it could be interpreted as meaning that~$I$ is generated by $I_+$ (as an ideal), that is, $I$ is equal to the smallest ideal containing~$I_+$. 
\end{pgr}

%==========================================================================================
Of particular importance are the positively spanned, hereditary ideals $I$ such that $I_+$ is also strongly invariant.
Such ideals were first considered by Dixmier, which justifies the following definition.

%==========================================================================================
\begin{dfn}
\label{dfn:DixmierIdeal}
A \emph{Dixmier ideal} in a \ca{} is a (not necessarily norm-closed) positively spanned, hereditary, two-sided ideal $I$ such that $I_+$ is strongly invariant.
\end{dfn}

%==========================================================================================
Dixmier was the first to study ideals as in \autoref{dfn:DixmierIdeal}.
Indeed, in his 1957 book (with English translation \cite{Dix81VNA}), he showed that every ideal in a von Neumann algebra is positively spanned and hereditary, and it follows easily that every ideal in a von Neumann algebra is a Dixmier ideal in the sense of \autoref{dfn:DixmierIdeal};
see \autoref{prp:IdealsVNA}.
Dixmier also considered positively spanned, hereditary ideals with strongly invariant positive cone in the setting of \ca{s} in Section~4.5 in \cite{Dix77CAlgebras}.

For von Neumann algebras, Dixmier even developed a theory of powers of ideals \cite{Dix52Remarques}.
Our theory of powers of Dixmier ideals can be viewed as the generalization of his theory to the \ca{ic} setting.

In \cite{Ror19FxPtConeTraces}, R{\o}rdam calls an ideal in a \ca{} \emph{symmetric} if its positive part is strongly invariant, and he systematically studies traces on symmetric, heredtiary ideals.
Note that an ideal in a \ca{} is Dixmier if and only if it is hereditary, symmetric, and spanned by its positive part;
see also \cite[Lemma~2.3]{Ror19FxPtConeTraces}.

For future reference, we note the following consequence of \autoref{pgr:OrderIdeal} for Dixmier ideals. 

%==========================================================================================
\begin{prp}
\label{prp:DixmierVsStrInvOrderIdeal}
Let $A$ be a \ca.
Then Dixmier ideals in~$A$ are in natural one-to-one correspondence with strongly invariant order ideals in $A_+$.
In one direction, given a Dixmier ideal $I \subseteq A$, the positive part $I_+$ is a strongly invariant order ideal such that $I = \linspan(I_+)$.
Conversely, given a strongly invariant order ideal $M \subseteq A_+$, the linear span $\linspan(M)$ is a Dixmier ideal such that $M = \linspan(M)_+$.
\end{prp}
\begin{proof}
Consider the following diagram:
\[
\xymatrix@R-5pt{
\left\{\parbox{3.3cm}{positively spanned, \\ hereditary ideals in $A$}\right\} \ar@{<->}[r] \ar@{}[d]|{\subseteqRotatedUp}
& \Big\{ \text{invariant order ideals in $A_+$} \Big\} \ar@{}[d]|{\subseteqRotatedUp} \\
\Big\{ \text{Dixmier ideals in $A$} \Big\} \ar@{<->}[r]
& \left\{\parbox{2.8cm}{strongly invariant \\ order ideals in $A_+$} \right\}
}
\]
The bijection in the top row is given by the assignments $I \mapsto I_+$ and $M \mapsto \linspan(M)$;
see \autoref{pgr:OrderIdeal} and \cite[Corollary~1.2]{Ped66MsrThyCAlg1}.
It follows from the definition of Dixmier ideals that this restricts to the bijection in the bottom row;
see also \cite[Theorem~1.1]{PedPet70IdealsCAlg}.
\end{proof}

%==========================================================================================
It essentially follows from results of Dixmier that ideals in von Neumann algebras are automatically Dixmier ideals.

%==========================================================================================
\begin{prp}[Dixmier]
\label{prp:IdealsVNA}
Every ideal in a von Neumann algebra is Dixmier.
\end{prp}
\begin{proof}
Let $I$ be an ideal in a von Neumann algebra $M$.
By Propositions~1.6.9 and~1.6.10 in \cite{Dix81VNA}, $I$ is positively spanned, and $I_+$ is an invariant order ideal in $M_+$.
To see that $I_+$ is also strongly invariant, let $x \in M$ such that $x^*x \in I_+$.
Consider the polar decomposition $x=v|x|$ in $M$.
Then $xx^* = v(x^*x)v^* \in I_+$.
\end{proof}

%==========================================================================================
It is also easy to see that norm-closed ideals are Dixmier ideals, using that \ca{s}
are semisimple (the intersection of their primitive ideals is trivial).
An alternative way to see this is to use that norm-closed ideals are known to be semiprime, in combination with our result \autoref{prp:SemiprimeDixmier} that semiprime ideals are Dixmier.
We illustrate an example of a non-Dixmier ideal that is very close to a Dixmier ideal. 
We will show in \autoref{prp:IntExt} that this is a general phenomenon.

%==========================================================================================
\begin{exa}
\label{exa:NotDixmier}
Consider $A=C([-1,1])$ and the principal ideal $I=fA$ generated by the function $f \in A$ with $f(t)=t$ for $t \in [0,1]$ and $f(t)=it$ for $t \in [-1,0]$.
Then~$I$ is not self-adjoint, since $f^* \notin I$;
see \cite[Examples~II.5.2.1(i)]{Bla06OpAlgs}.
Thus, $I$ is not a Dixmier ideal.

Let $\varepsilon\geq 0$, and set
\[
J_\varepsilon := \big\{ g \in A : |g| \leq m|f|^{1+\varepsilon} \text{ for some $m\in\NN$} \big\}.
\]
Then $J_\varepsilon$ is a Dixmier ideal: indeed, its positive cone
$(J_\varepsilon)_+$ is readily seen to be an order ideal, and it
is automatically strongly invariant because $A$ is commutative.
Moreover, we have $J_\varepsilon \subseteq I \subseteq J_0$ for every $\varepsilon>0$.
\end{exa}

%==========================================================================================
The intersection of a family of strongly invariant order ideals is again a strongly invariant order ideal.
We may therefore speak of the strongly invariant order ideal generated by a positive element.

%==========================================================================================
\begin{prp}
\label{prp:LHD}
Let $A$ be a \ca, and let $a,b \in A_+$.
Then $a \lhd b$ if and only if $a$ belongs to the strongly invariant order ideal in $A_+$ generated by $b$.
\end{prp}
\begin{proof}
Consider the set $M := \{c \in A_+ : c \lhd b\}$.
It is easy to see that every element in $M$ is contained in the strongly invariant order ideal generated by $b$.
To show the converse inclusion, it suffices to show that $M$ is a strongly invariant order ideal containing $b$.
	
Since $\lhd$ is reflexive, we have $b \in M$.
To see that $M$ is closed under addition, let $c_1,c_2 \in M$.
We get $c_1+c_2 \lhd 2b$ by \autoref{prp:PropertiesRelations}(4), and thus $c_1+c_2 \lhd b$ by definition of~$\lhd$.
Thus, $c_1+c_2 \in M$.
Using that $\leq$ and $\approx$ are stronger than $\lessapprox$ and $\lessapprox_+$, and hence than $\lhd$, we see that $M$ is a strongly invariant order ideal.
\end{proof}

%==========================================================================================
As a consequence of \autoref{prp:PowerLhd} we obtain:

%==========================================================================================
\begin{prp}
\label{prp:PowerOrderIdeal}
Let $A$ be a \ca, let $M \subseteq A_+$ be a strongly invariant order ideal, and let $s \in (0,\infty)$.
Then
\[
M^s := \big\{ a^s : a \in M \big\}
\]
is a strongly invariant order ideal.
\end{prp}
\begin{proof}
Let $N$ denote the strongly invariant order ideal generated by $M^s$.
It follows from \autoref{prp:LHD} that
\[
N = \big\{ b \in A_+ : b \lhd a_1^s + \ldots + a_n^s \text{ for some } a_1, \ldots, a_n \in M \big\}.
\]
By \autoref{prp:PowerSumLhd}, we have $a_1^s + \ldots + a_n^s \lhd (a_1+\ldots+a_n)^s$ and thus
\[
N = \big\{ b \in A_+ : b \lhd a^s \text{ for some } a \in M \big\}.
\]

To show that $N \subseteq M^s$, let $b \in N$.
Choose $a \in M$ such that $b \lhd a^s$.
By \autoref{prp:PowerLhd}, we get $b^{\frac{1}{s}} \lhd a$. 
Since $M$ is a strongly invariant order ideal, this implies that $b^{\frac{1}{s}} \in M$, and thus $b \in M^s$.
\end{proof}

%==========================================================================================
Taking advantage of the natural correspondence between strongly invariant order ideals in $A_+$ and Dixmier ideals in $A$ described in \autoref{prp:DixmierVsStrInvOrderIdeal}, we next define powers $I^s$ of a Dixmier ideal $I$ for every exponent $s \in (0,\infty)$.
It follows from \autoref{prp:PowersDixmierIdeal} below that for every integer $n$, this `new' definition of $I^n$ agrees with the `ring theoretic' definition of $I^n$ as the ideal generated by the set of products $x_1 \cdot\ldots\cdot x_n$ for $x_1,\ldots,x_n \in I$.

%==========================================================================================
\begin{dfn}
\label{dfn:PowerDixmierIdeal}
Let $I$ be a Dixmier ideal in a \ca{}.
Given $s \in (0,\infty)$, we define the $s$-th power of $I$, or the $\frac{1}{s}$-th root of $I$, to be
\[
I^s := \linspan \big((I_+)^s\big) = \linspan\big(\big\{ a^s : a \in I_+ \big\}\big).
\]
\end{dfn}

%==========================================================================================
\begin{thm}
\label{prp:PowersDixmierIdeal}
Let $I$ be a Dixmier ideal in a \ca{} $A$.
Then:
\begin{enumerate}
\item
For every $s \in (0,\infty)$, $I^s$ is a Dixmier ideal with $(I^s)_+ = \{a^s : a\in I_+\}$.
\item
We have $(I^s)^t=I^{st}=(I^t)^s$ for all $s,t\in(0,\infty)$.
\item
We have $I^s I^t = I^{s+t}$ for all $s,t\in(0,\infty)$.
\end{enumerate}
In particular, we have $I^s \subseteq I^t$ whenever $s \geq t > 0$.
\end{thm}

%==========================================================================================
Note that if $I$ is a Dixmier ideal and $s \in (0,\infty)$, then the positive part of $I^s$ agrees with the set $\{a^s : a \in I_+\}$ by (1), and it is unambiguous to write $I^s_+$.

%==========================================================================================
\begin{proof}
(1). 
It follows immediately from \autoref{prp:PowerOrderIdeal} and \autoref{prp:DixmierVsStrInvOrderIdeal} that~$I^s$ is a Dixmier ideal with $(I^s)_+=\{a^s\colon a\in I_+\}$.

To verify~(2), let $s,t \in (0,\infty)$.
Using~(1), we get
\[
((I^s)^t)_+
= \big\{ a^t : a \in (I^s)_+ \big\}
= \big\{ (a^s)^t : a \in I_+ \big\}
= \big\{ a^{st} : a \in I_+ \big\}
= (I^{st})_+.
\]
which implies that $(I^s)^t=I^{st}$, and then also $(I^t)^s=I^{st}$ by exchanging $s$ and $t$.

We are going to show (3). 
Without loss of generality, we may assume that $A$ is unital.
During the proof, we temporarily write $I^0:=A$, which is positively spanned by $I^0_+:=A_+$. 

\medskip

\textbf{Claim 1:} 
\emph{Given $s,t \in [0,\infty)$, we have $I^{s+t} \subseteq I^sI^t$.}
This is clear if $s$ or $t$ are zero.
To verify the claim for $s,t>0$, let $a \in I^{s+t}_+$.
By~(1), we have $b := a^{\frac{1}{s+t}} \in I_+$.
Then $b^s \in I^s$ and $b^t \in I^t$, and thus
\[
a = \big(a^{\frac{1}{s+t}}\big)^s\big(a^{\frac{1}{s+t}}\big)^t = b^sb^t \in I^s I^t.
\]
Using that $I^{s+t}$ is spanned by its positive elements, we get $I^{s+t} \subseteq I^sI^t$, which verifies the claim.

\medskip

\textbf{Claim 2:} 
\emph{Given $s \in (0,\frac{1}{2})$ and $r\in[0,1)$, we have $I^sI^rI^s \subseteq I^{2s+r}$.}
%We proceed similarly as in the proof of Claim~2.
By a variant of polarization, \cite[Proposition~II.3.1.9(i)]{Bla06OpAlgs}, we have 
\[
x^*by
= \sum_{k=0}^3 \tfrac{(-i)^k}{4} (x+i^ky)^*b(x+i^ky)
\]
for all $x,y \in I^s$ and $b \in I^r_+$.
Using that $I^{s}$ is self-adjoint, that $I^r$ is positively spanned, and that $I^{2s+r}$ is closed under scalar multiplication, it therefore suffices to show that $x^*bx \in I^{2s+r}$ for every $x \in I^s$ and $b \in I^r_+$.

Let $x \in I^{s}$ and $b \in I^r_+$.
Using that $I^s$ is positively spanned, we have $x = \sum_{k=0}^3 i^k a_k$ for suitable $a_k \in I^s_+$.
Using that $e^*bf+f^*be \leq e^*be+f^*bf$ for any $e,f \in A$, we get
\[
x^*bx
= \sum_{k,l=0}^3 (i^ka_k)^*b(i^la_l)
\leq 4 \sum_{k=0}^3 a_k^*ba_k.
\]

Since $I^{2s+r}_+\subseteq A$ is an order ideal, it therefore suffices to show that $aba \in I^{2s+r}$ for all $a \in I^s_+$ and $b \in I^r_+$.
Let $a_0,b_0 \in I_+$.
Using that $b_0 \leq a_0+b_0$, and that $t\mapsto t^r$ and $t \mapsto t^{2s}$ are operator monotone, we get
\begin{align*}
a_0^s b_0^r a_0^s
&\leq a_0^s (a_0+b_0)^r a_0^s
\approx (a_0+b_0)^{\frac{r}{2}}a_0^{2s}(a_0+b_0)^{\frac{r}{2}} \\
&\leq (a_0+b_0)^{\frac{r}{2}}(a_0+b_0)^{2s}(a_0+b_0)^{\frac{r}{2}}
= (a_0+b_0)^{2s+r}.
\end{align*}

We have $a_0+b_0 \in I_+$, and thus $(a_0+b_0)^{2s+r} \in I^{2s+r}_+$ by~(1).
Using that~$I^{2s+r}_+$ is an order ideal, we see that $aba \in I^{2s+r}$ for all $a \in I^s_+$ and $b \in I^r_+$.
This proves the claim.

\medskip

\textbf{Claim 3:} 
\emph{Given $s,t \in (0,\frac{1}{2})$, we have $I^sI^t \subseteq I^{s+t}$.}
If $s\leq t$, then set $r := t-s$.
Using Claim~1 at the second step, and Claim~2 at the third step, we have
\[
I^sI^t
= I^s I^{r+s}
\subseteq I^s I^r I^s
\subseteq I^{2s+r}
= I^{s+t}.
\]
The case $s>t$ is shown analogously.

\medskip

\textbf{Claim 4:} 
\emph{Given $s,t \in (0,\infty)$, we have $I^sI^t \subseteq I^{s+t}$.}
Choose $r\in (0,\infty)$ big enough such that $\frac{s}{r}, \frac{t}{r} < \frac{1}{2}$. 
Using~(2) at the first and last step, and using Claim~3 for the Dixmier ideal $I^r$ at the second step, we have 
\[
I^sI^t
=(I^r)^{\frac{s}{r}} (I^r)^{\frac{t}{r}}
\subseteq (I^{r})^{\frac{s}{r}+\frac{t}{r}}
= I^{s+t}.
\]
This proves the claim.

Combining Claim~1 and Claim~4, we have $I^s I^t = I^{s+t}$ for all $s,t\in(0,\infty)$.
\end{proof}

%==========================================================================================
\begin{prp}
\label{prp:CharRootDixmierIdeal}
Let $I$ be a Dixmier ideal in a \ca{} $A$.
Then
\[
I^{\frac{1}{2}} 
= \big\{ x \in A : x^*x \in I \big\}
= \big\{ x \in A : xx^* \in I \big\}.
\]
\end{prp}
\begin{proof}
Since $I_+$ is strongly invariant, we have
\[
J 
:= \big\{ x \in A : x^*x \in I \big\}
= \big\{ x \in A : xx^* \in I \big\}.
\]
In particular, an element $x \in A$ satisfies $x \in J$ if and only if $x^* \in J$.

To verify that $I^{\frac{1}{2}} \subseteq J$, let $x \in I^{\frac{1}{2}}$.
Using that~$I^{\frac{1}{2}}$ is Dixmier and thus self-adjoint, we have $x^* \in I^{\frac{1}{2}}$, and using \autoref{prp:PowersDixmierIdeal}~(3), we get
\[
x^*x 
\in I^{\frac{1}{2}}I^{\frac{1}{2}} 
= I.
\]
Thus $x$ belongs to $J$, as desired.

Next, we show that $J \subseteq I^{\frac{1}{2}}$.
To verify that $J$ is a linear subspace, let $x,y \in J$ and $\lambda\in\CC$.
Then
\[
(\lambda x + y)^*(\lambda x + y)
\leq 2|\lambda|^2x^*x + 2 y^*y \in I.
\]
Using that $I_+$ is an order ideal, we see that $\lambda x + y \in J$, as desired.
It therefore suffices to show that every self-adjoint element in $J$ is contained in~$I^{\frac{1}{2}}$.
Let $x \in J$ be self-adjoint.
Let $x_+ \in A_+$ and $x_- \in A_+$ be the positive and negative parts of $x$.
Then
\[
x_+^2 \leq x_+^2 + x_-^2 = x^*x \in I_+.
\]
Using that $I_+$ is an order ideal, it follows that $x_+^2 \in I_+$, and so $x_+ \in I^{\frac{1}{2}}_+$.
Analogously, we get $x_- \in I^{\frac{1}{2}}_+$, and in conclusion $x = x_+-x_- \in I^{\frac{1}{2}}$.
\end{proof}

%==========================================================================================
Next, we observe that a Dixmier ideal has a unique root among Dixmier ideals.
This is no longer true among all ideals:
\autoref{exa:NonDixmierRoot} below shows that there exist ideals that are not Dixmier, but that have a Dixmier square.

%==========================================================================================
\begin{cor}
\label{prp:UniqueDixmierRoot}
Let $I$ be a Dixmier ideal in a \ca{} $A$, let $n \in \NN$ with $n \geq 1$, and let $J$ be a Dixmier ideal such that $J^n = I$.
Then $J = I^{\frac{1}{n}}$.
\end{cor}
\begin{proof}
Using \autoref{prp:PowersDixmierIdeal}~(2), we have
$I^{\frac{1}{n}}
= (J^n)^{\frac{1}{n}}
= J$, 
as desired.
\end{proof}

%==========================================================================================
\begin{exa}
\label{exa:NonDixmierRoot}
Consider the \ca{} $C_0(\NN,M_2(\CC))$ of sequences of $2$-by-$2$ matrices $\begin{psmallmatrix} a_n & b_n \\ c_n & d_n \end{psmallmatrix}$ such that $(a_n)_{n\in\NN}$, $(b_n)_{n\in\NN}$, $(c_n)_{n\in\NN}$ and $(d_n)_{n\in\NN}$ converge to zero.
Let~$I$ be the ideal of such matrices with $|a_n|,|b_n|,|c_n| = O(\tfrac{1}{n})$ and $|d_n| = o(\tfrac{1}{n})$.
Then $I$ is a self-adjoint ideal that is not Dixmier:
for the sequence $x=(x_n)_{n\in\NN}$ with $x_n=\begin{psmallmatrix} 0 & 0 \\ n^{-\frac{1}{2}} & 0 \end{psmallmatrix}$, we have $x^*x \in I$ but $xx^* \notin I$.
We refer to \cite[Proposition~1.3]{Ped69MsrThyCAlg34} for the original version of this example; see also \cite[Examples~II.5.2.1(vi)]{Bla06OpAlgs}.

Set $J:=I^2$, which is the ideal consisting of sequences of matrices $\begin{psmallmatrix} a_n & b_n \\ c_n & d_n \end{psmallmatrix}$ with  $|a_n|,|b_n|,|c_n|,|d_n| = O(\tfrac{1}{n^2})$.
Then $J$ is a Dixmier ideal, whose (Dixmier) square-root $J^{\frac{1}{2}}$ is the ideal consisting of sequences of matrices with  $|a_n|,|b_n|,|c_n|,|d_n| = O(\tfrac{1}{n})$.
Then $(J^{\frac{1}{2}})^2 = J = I^2$, but $J^{\frac{1}{2}} \neq I$.
\end{exa}

%==========================================================================================
%==========================================================================================
\section{The lattice of Dixmier ideals}
\label{sec:Lattice}

%==========================================================================================
We observe that the Dixmier ideals in a \ca{} $A$ form a complete sublattice of the lattice of ideals in $A$;
see \autoref{prp:LatticeDixmierIdeals}.
It follows that every ideal $I$ is contained in a smallest Dixmier ideal -- the `Dixmier closure' $\exterior(I)$ of $I$ -- and contains a largest Dixmier ideal -- the `Dixmier interior' $\interior(I)$ of $I$.
We show that the inclusions $\interior(I)\subseteq I\subseteq \exterior(I)$ are very tight in a suitable sense;
see \autoref{prp:IntExt}.

\medskip

%==========================================================================================
It is well-known that the family of (not necessarily norm-closed) ideals in a \ca{} forms a complete lattice.
The infimum of a family $(I_\lambda)_{\lambda\in\Lambda}$ of ideals is their intersection, $\bigwedge_{\lambda\in\Lambda} I_\lambda = \bigcap_{\lambda\in\Lambda} I_\lambda$, and the supremum is $\bigvee_{\lambda\in\Lambda} I_\lambda = \sum_{\lambda\in\Lambda} I_\lambda$, where $\sum_{\lambda\in\Lambda} I_\lambda$ is defined as the set of finite sums of elements in $\bigcup_{\lambda\in\Lambda} I_\lambda$.

%==========================================================================================
\begin{prp}
\label{prp:LatticeDixmierIdeals}
The Dixmier ideals of a \ca{} $A$ form a complete sublattice of the ideal lattice of $A$.
Further, given a family $(I_\lambda)_{\lambda \in \Lambda}$ of Dixmier ideals, we have
\[
\Big( \bigcap_{\lambda\in\Lambda} I_\lambda \Big)_+ 
= \bigcap_{\lambda\in\Lambda} \big(I_\lambda\big)_+, \andSep
\Big( \sum_{\lambda\in\Lambda} I_\lambda \Big)_+ 
= \sum_{\lambda\in\Lambda} \big(I_\lambda)_+.
\]

Given also $s \in (0,\infty)$, we have
\[
\Big( \bigcap_{\lambda\in\Lambda} I_\lambda \Big)^s 
= \bigcap_{\lambda\in\Lambda} I_\lambda^s, \andSep
\Big( \sum_{\lambda\in\Lambda} I_\lambda \Big)^s
= \sum_{\lambda\in\Lambda} I_\lambda^s.
\]
\end{prp}
\begin{proof}
Using that the intersection of a family of strongly invariant order ideals in~$A_+$ is again a strongly invariant order ideal, it follows that $\bigcap_{\lambda\in\Lambda} I_\lambda$ is again a Dixmier ideal with positive part $\bigcap_{\lambda\in\Lambda} (I_\lambda)_+$.

Next, we consider the supremum.
By \cite[Proposition~1.5]{Ped69MsrThyCAlg34}, $M+N$ is a strongly invariant order ideal whenever $M,N \subseteq A_+$ are.
This also follows from \autoref{prp:LHD} and the Riesz decomposition property for the relation $\lhd$ by \autoref{prp:PropertiesRelations}. 
Therefore, given a finite subset $\Lambda' \subseteq \Lambda$, the sum $\sum_{\lambda \in \Lambda'} I_\lambda$ is a Dixmier ideal with positive part $\sum_{\lambda \in \Lambda'} (I_\lambda)_+$.
Using that $\sum_{\lambda \in \Lambda} I_\lambda$ is the union of the directed family of Dixmier ideals $\sum_{\lambda \in \Lambda'} I_\lambda$ indexed over the finite subsets $\Lambda'$ of $\Lambda$, it follows that $\sum_{\lambda \in \Lambda} I_\lambda$ is a Dixmier ideal with positive part $\sum_{\lambda\in\Lambda} (I_\lambda)_+$.

To show compatibility with taking powers, let $s \in (0,\infty)$.
Using the first part of the argument and \autoref{prp:PowersDixmierIdeal}~(1), we have
\[
\Big( \sum_{\lambda\in\Lambda} I_\lambda \Big)^s_+
= \biggl\{ a^s : a \in \Big( \sum_{\lambda\in\Lambda} I_\lambda \Big)_+ \biggr\}
= \biggl\{ a^s : a \in \sum_{\lambda\in\Lambda} \big(I_\lambda\big)_+ \biggr\}
= \sum_{\lambda\in\Lambda} \big(I_\lambda^s\big)_+,
\]
where for the last step we have used the fact that for any finite subset $\Lambda'\subseteq \Lambda$ and $a_\lambda\in (I_\lambda)_+$ for $\lambda\in\Lambda'$, it holds $\sum_{\lambda\in\Lambda'}a_\lambda^s \lhd \bigl(\sum_{\lambda\in\Lambda'}a_\lambda\bigr)^s$ and $\bigl(\sum_{\lambda\in\Lambda'}a_\lambda\bigr)^s \lhd \sum_{\lambda\in\Lambda'}a_\lambda^s$ by \autoref{prp:PowerSumLhd}. 
It follows that $\big( \sum_{\lambda\in\Lambda} I_\lambda \big)^s = \sum_{\lambda\in\Lambda} I_\lambda^s$.
The compatibility of infima with taking powers follows analogously.
\end{proof}

%==========================================================================================
\begin{dfn}
\label{dfn:DixmierIntExt}
Given an ideal $I \subseteq A$ in a \ca{}, its \emph{Dixmier interior} is
\[
\interior(I) := \bigvee \big\{ J \subseteq A : J \text{ Dixmier ideal}, J \subseteq I \big\},
\]
and its \emph{Dixmier closure} is
\[
\exterior(I) := \bigwedge \big\{ J \subseteq A : J \text{ Dixmier ideal}, I \subseteq J \big\}.
\]
\end{dfn}

%==========================================================================================
\begin{lma}
\label{prp:ApproxDominated}
Let $I\subseteq A$ be an ideal in a \ca\ $A$, and let $a\in A_+$ be such that there exists $x \in I$ and $\varepsilon>0$ with $a \lhd |x|^{1+\varepsilon}$.
Then $\lambda a \in I$ for all $\lambda \in \CC$.
\end{lma}
\begin{proof}
We divide the proof into several parts.
\vspace{.1cm}

\textbf{Claim 1:} 
\emph{Given $x \in I$ and $\varepsilon \in (0,1)$, we have $|x|^{1+\varepsilon} \in I$.}
To prove the claim, we apply Polar Decomposition (\autoref{prp:PolarRiesz}~(1)) to obtain a partial isometry $v \in A^{**}$ such that $x = v|x|$, and such that $u:=v|x|^\varepsilon$ belongs to $A$.
Since $v^*v$ is the right support projection of $x$, we get
\[
x = u|x|^{1-\varepsilon}, \andSep 
u^*u 
= |x|^\varepsilon v^*v |x|^\varepsilon
= |x|^{2\varepsilon}.
\]

Then 
\[
|x|^{1+\varepsilon} 
= |x|^{2\varepsilon}|x|^{1-\varepsilon} 
= u^*u |x|^{1-\varepsilon} 
= u^*x,
\]
and since $I$ is an ideal, we get $|x|^{1+\varepsilon} \in I$.
This proves the claim.

\medskip

\textbf{Claim 2:} 
\emph{Given $\lambda \in \CC$, $a \in A_+$, $x \in I$ and $\varepsilon > 0$ with $a \lessapprox |x|^{1+\varepsilon}$, we have $\lambda a \in I$.}
To prove the claim, pick $\varepsilon'>0$ and $\varepsilon''\in(0,1)$ such that $1+\varepsilon = (1+\varepsilon')(1+\varepsilon'')$, and set $b := |x|^{1+\varepsilon''}$.
By Claim~1, we have $b \in I_+$.
Further, we have $a \lessapprox |x|^{1+\varepsilon} = b^{1+\varepsilon'}$.

Choose $y \in A$ such that
\[
a=yy^*, \andSep 
y^*y \leq b^{1+\varepsilon'}.
\]
Using Polar Decomposition (\autoref{prp:PolarRiesz}~(2)), we obtain $w \in A^{**}$ such that $y = wb^{\frac{1+\varepsilon'}{2}}$, and such that $z := wb^{\frac{\varepsilon'}{2}}$ belongs to $A$.
Then $y = zb^{\frac{1}{2}}$ and thus
\[
\lambda a 
= \lambda yy^* 
= \lambda zb^{\frac{1}{2}}(zb^{\frac{1}{2}})^* 
= (\lambda z)bz^*,
\]
and since $I$ is an ideal, we get $\lambda a \in I$, as desired.

\medskip

To prove the statement, let $\lambda \in \CC$, $a \in A_+$, $x \in I$ and $\varepsilon>0$ such that $a \lhd |x|^{1+\varepsilon}$.
By definition, we obtain $n,m\in\NN$ with $n \geq 1$, and $a_1,\ldots,a_n, b_1,\ldots,b_n \in A_+$ such that
\[
a = \sum_{j=1}^n a_j, \quad
a_j \approx b_j \quad \text{for $j=1,\ldots,n$}, \andSep
\sum_{j=1}^nb_j \leq m|x|^{1+\varepsilon}.
\]

For each $j$, we have
\[
a_j \approx b_j \leq m|x|^{1+\varepsilon} \leq |mx|^{1+\varepsilon}.
\]
Using that $mx \in I$, we can apply Claim~2 to deduce that $\lambda a_j \in I$.
It follows that $\lambda a = \sum_{j=1}^n \lambda a_j \in I$, as desired.
\end{proof}

%==========================================================================================
\begin{thm}
\label{prp:IntExt}
Let $I \subseteq A$ be an ideal in a \ca\ $A$.
Then $|x|^{1+\varepsilon} \in \interior(I)$ for all $x \in I$ and $\varepsilon>0$.
Further, we have
\[
\exterior(I)^{1+\varepsilon} 
\subseteq \interior(I)
\subseteq I 
\subseteq \exterior(I)
\subseteq \interior(I)^{1-\delta}
\]
for all $\varepsilon>0$ and $\delta\in (0,1)$.
\end{thm}
\begin{proof}
Let $x \in I$ and $\varepsilon>0$.
Set
\[
M := \big\{ a \in A_+ : a \lhd |x|^{1+\varepsilon} \big\},
\]
which by \autoref{prp:LHD} is the smallest strongly invariant order ideal of~$A_+$ containing~$|x|^{1+\varepsilon}$.
Thus the linear span $\linspan(M)$ of $M$ is a Dixmier ideal.
It follows from \autoref{prp:ApproxDominated} that $\linspan(M)$ is contained in $I$, whence
\[
|x|^{1+\varepsilon} 
\in M 
\subseteq \linspan(M) 
\subseteq \interior(I).
\]

Let us now prove the second part of the statement.
It is clear that $\interior(I) \subseteq I \subseteq \exterior(I)$.
We have
\[
\exterior(I)_+ = \big\{ a \in A_+ : a \lhd b \text{ for some $b \in I_+$} \big\}.
\]

Let $\varepsilon>0$.
By \autoref{prp:PowerLhd}, two elements $a,b \in A_+$ satisfy $a \lhd b$ if and only if $a^{1+\varepsilon} \lhd b^{1+\varepsilon}$.
Using this at the first step, and using the first part of the statement at the second step, we get
\[
\exterior(I)^{1+\varepsilon}_+ 
= \big\{ a \in A_+ : a \lhd b^{1+\varepsilon} \text{ for some $b \in I_+$} \big\}
\subseteq \interior(I)_+.
\]
It follows that $\exterior(I)^{1+\varepsilon} \subseteq \interior(I)$.

Given $\delta\in(0,1)$, to verify that $\exterior(I) \subseteq \interior(I)^{1-\delta}$, set $\varepsilon = (1-\delta)^{-1}-1$.
Then $\varepsilon>0$ and $1=(1-\delta)(1+\varepsilon)$.
By the above argument, we have $\exterior(I)^{1+\varepsilon} \subseteq \interior(I)$.
Applying \autoref{prp:PowersDixmierIdeal}, we then get
\[
\exterior(I)
= \big(\exterior(I)^{1+\varepsilon}\big)^{1-\delta}
\subseteq \big(\interior(I)\big)^{1-\delta},
\]
which finishes the proof.
\end{proof}

%==========================================================================================
%==========================================================================================
\section{Semiprime ideals}
\label{sec:SemiprimeIdeals}

%==========================================================================================
In this section, we use our theory of powers and roots of Dixmier ideals from \autoref{sec:DixmierIdeals} to prove the main result of the paper:
A not necessarily norm-closed ideal in a \ca{} is semiprime if and only if it is idempotent, if and only if it is closed under roots of positive elements;
see \autoref{prp:CharSemiprime}.
We also derive the several consequences of \autoref{prp:CharSemiprime} listed in the introduction.

%==========================================================================================
\begin{lma}
\label{prp:CharSemiprimeDixmier}
Let $I \subseteq A$ be a Dixmier ideal in a \ca.
Then the following are equivalent:
\begin{enumerate}
\item
$I = I^s$ for all $s \in (0,\infty)$.
\item
$I = I^2$.
\item
$I = I^{\frac{1}{2}}$.
\item
$I^s = I^t$ for some $s,t \in (0,\infty)$ with $s \neq t$.
\end{enumerate}
\end{lma}
\begin{proof}
It is clear that~(1) implies~(2), and that~(3) implies~(4).
It follows from \autoref{prp:PowersDixmierIdeal} that~(2) implies~(3).
Let us show that~(4) implies~(1).
By assumption, there exist distinct $s,t \in (0,\infty)$ with $I^s = I^t$.
Without loss of generality, we may assume that $s < t$.
Applying \autoref{prp:PowersDixmierIdeal}, we get
\[
I 
= \big(I^s\big)^{\frac{1}{s}} 
= \big(I^t\big)^{\frac{1}{s}}
= I^{\frac{t}{s}}.
\]
Applying \autoref{prp:PowersDixmierIdeal} again at the last step, we get
\[
I 
= I^{\frac{t}{s}}
= \big(I\big)^{\frac{t}{s}}
= \big(I^{\frac{t}{s}}\big)^{\frac{t}{s}}
= I^{(\frac{t}{s})^2}.
\]
Inductively, we get $I = I^{(\frac{t}{s})^n}$ for all $n \geq 1$.

Given any $r \in [1,\infty)$, pick $n$ with $r \leq (\frac{t}{s})^n$.
By \autoref{prp:PowersDixmierIdeal}, we have $I^{r_1} \subseteq I^{r_2}$ whenever $r_1 \geq r_2$.
Hence,
\[
I 
= I^{(\frac{t}{s})^n}
\subseteq I^r
\subseteq I.
\]
It follows that $I=I^r$ for all $r \in [1,\infty)$.
For $r\in(0,1]$, we use that $I=I^{\frac{1}{r}}$ and apply \autoref{prp:PowersDixmierIdeal} to get
\[
I = \big(I^{\frac{1}{r}}\big)^r = I^r,
\]
which completes the verification of~(1).
\end{proof}

%==========================================================================================
Following the terminology in Banach algebras, we say that an ideal $I$ in a \ca{} \emph{factors} if for every $x \in I$ there are $y,z \in I$ such that $x=zy$.
Further, $I$ is said to \emph{factor weakly} if $I$ is the linear span of products $xy$ for $x,y \in I$;
see, for example, \cite{DalFeiPha21FactCommBAlg}.

Recall that $I$ is \emph{idempotent} if $I=I^2$, that is, every $x \in I$ is of the form $x = y_1z_1+\ldots+y_nz_n$ for some $n$ and $y_j,z_j \in I$.
In this case, $\lambda x=\sum_j(\lambda y_j)z_j\in I$ for all $\lambda\in\CC$, and we see that idempotent ideals are automatically closed under scalar multiplication.
This shows that $I$ is idempotent if and only if it factors weakly.

%==========================================================================================
\begin{thm}
\label{prp:CharSemiprime}
Let $I \subseteq A$ be an ideal in a \ca.
Then the following are equivalent:
\begin{enumerate}
\item
$I$ is semiprime.
\item
$I$ is idempotent (factors weakly), that is, $I = I^2$.
\item
$I^m = I^n$ for some $m \neq n$.
\item
$I$ factors. %: For every $x \in I$ exists $y,z \in I$ such that $x = yz$.
\item
$I$ is closed under arbitrary roots of positive elements.
\item
$I$ is closed under square roots of positive elements.
\item
$\interior(I)$ is semiprime.
\item
$\exterior(I)$ is semiprime.
\end{enumerate}
If the above conditions hold, then $I$ is Dixmier and thus $\interior(I)=I=\exterior(I)$.
\end{thm}
\begin{proof}
To simplify notation, set $J := \interior(I)$ and $K := \exterior(I)$.
Then $J$ and $K$ are Dixmier ideals, and by \autoref{prp:IntExt} for every $\varepsilon\in(0,1)$ we have
\[
K^{1+\varepsilon} \subseteq J \subseteq I \subseteq K \subseteq J^{1-\varepsilon}.
\]

\medskip

\emph{We show that~(1) implies~(2) and $I=K$.}
Assume that $I$ is semiprime. 
We have
\[
(K^{\frac{2}{3}}\big)^2 
= K^\frac{4}{3}
\subseteq I.
\]
By assumption, it follows that $K^{\frac{2}{3}} \subseteq I$.
Then $K \subseteq K^{\frac{2}{3}} \subseteq I \subseteq K$, and thus $K^{\frac{2}{3}} = K = I$.
By \autoref{prp:CharSemiprimeDixmier}, we have $K=K^2$ and thus $I = I^2$.

\medskip

\emph{We show that~(3) implies~(1).}
Let $m, n \geq 1$ with $m \neq n$ and $I^m = I^n$.
Without loss of generality, we may assume that $m < n$.
We have $I \subseteq J^{1-\frac{1}{2n}}$.
Applying also \autoref{prp:PowersDixmierIdeal}, we get
\[
I^m 
= I^n 
\subseteq \Big( J^{1-\frac{1}{2n}} \Big)^n 
= J^{n-\frac{1}{2}} 
\subseteq J^m \subseteq I^m,
\]
and thus $J^{n-\frac{1}{2}} = J^m$.
It follows from \autoref{prp:CharSemiprimeDixmier} that $J=J^s$ for all $s \in (0,\infty)$, and in particular $I=J=K$. We will use this for $s=\tfrac{1}{8}$ below.

To show that $I$ is semiprime, let $L \subseteq A$ be an ideal such that $L^2 \subseteq I$.
Then
\[
\interior(L)^2 
\subseteq L^2 
\subseteq I
\subseteq J^{\frac{1}{2}}.
\]
Applying \autoref{prp:IntExt} at the first step, and \autoref{prp:PowersDixmierIdeal} at the second and fourth step, we get
\[
L 
\subseteq \interior(L)^{\frac{1}{2}} 
= \big(\interior(L)^2\big)^\frac{1}{4}
\subseteq \big(J^{\frac{1}{2}}\big)^\frac{1}{4}
= J^{\frac{1}{8}}
= J
\subseteq I,
\]
as desired.

It is clear that~(2) implies~(3).
This shows that~(1), (2) and~(3) are equivalent and that each of them implies that $I$ is Dixmier.

\medskip

\emph{We show that~(1) implies~(5).}
Assume that $I$ is semiprime.
We have seen that every semiprime ideal is idempotent and Dixmier.
Thus, $I$ is Dixmier with $I=I^2$.
Then $I=I^{\frac{1}{n}}$ for every $n \geq 1$ by \autoref{prp:CharSemiprimeDixmier}, which shows that $I$ is closed under arbitrary roots of positive elements.

\medskip 

It is clear that~(5) implies (6).
\emph{We show that~(6) implies~(4).}
Let $x \in I$.
Applying polar decomposition (\autoref{prp:PolarRiesz}~(1)) for $x$, we obtain a partial isometry $v \in A^{**}$ such that $x = v|x|$ in $A^{**}$, and such that $y := v|x|^{\frac{1}{2}}$ belongs to~$A$.
Then
\[
x
= v|x|
= \big( v|x|^\frac{1}{2} \big) |x|^{\frac{1}{2}}
= y (x^*x)^\frac{1}{4}
= \big( y (x^*x)^\frac{1}{8} \big) (x^*x)^\frac{1}{8}.
\]
Since $I$ is an ideal, we have $x^*x \in I_+$.
Using the assumption, we get $(x^*x)^{\frac{1}{2}} \in I_+$, then $(x^*x)^\frac{1}{4} \in I_+$, and finally $(x^*x)^\frac{1}{8} \in I_+$.
Using again that $I$ is an ideal, it follows that $y(x^*x)^\frac{1}{8} \in I$, and so $x$ is the product of two elements in $I$.

It is clear that~(4) implies~(3).
This shows that (1)--(6) are equivalent.

\medskip

\emph{We show that~(6) implies~(7).}
Assume that $a^{\frac{1}{2}} \in I$ for every $a \in I_+$.
Since $J \subseteq I$, we deduce that $J^{\frac{1}{2}}_+ \subseteq I_+$.
We always have $I \subseteq J^{\frac{2}{3}}\subseteq J^{\frac{1}{2}}$, and thus $J^{\frac{1}{2}} = J^{\frac{2}{3}}$.
By \autoref{prp:CharSemiprimeDixmier}, we get $J=J^2$, and so $J$ is semiprime by the already established equivalence between~(1) and~(2).

\medskip

\emph{We show that~(7) implies~(8).}
Assume that $J$ is semiprime.
Then $J=J^2$ by the established equivalence between~(1) and~(2).
Applying \autoref{prp:CharSemiprimeDixmier}, we get $J=J^{\frac{1}{2}}$, and then
\[
J \subseteq I \subseteq K \subseteq J^{\frac{1}{2}} =J.
\]
This shows that $K=J$, and so $K$ is semiprime.

\medskip

\emph{We show that~(8) implies~(1).}
Assume that $K$ is semiprime.
Then $K=K^2$ by the established equivalence between~(1) and~(2).
Then
\[
K^2 \subseteq I \subseteq K = K^2.
\]
This shows that $I=K$, and so $I$ is semiprime.
\end{proof}

%==========================================================================================
\begin{rmks}
(1)
Ideals that are closed under square roots of positive elements were studied by Pedersen--Petersen in \cite{PedPet70IdealsCAlg}, who called them \emph{algebraic ideals}. 
Among other things, they showed in \cite[below Corollary~1.2]{PedPet70IdealsCAlg} that 
this property is equivalent to $a^t \in I_+$ for all $a \in I_+$ and $t>0$.
\autoref{prp:CharSemiprime} shows that the `algebraic ideals' studied by Pedersen--Petersen are precisely the semiprime ideals in \ca{s}.

(2)
A ring is said to be \emph{fully idempotent} if every ideal in it is idempotent, and
\emph{fully semiprime} if every ideal in it is semiprime. 
Courter showed in \cite{Cou69AllFactorRgsSemiprime} that a (possibly non-unital) ring is fully idempotent if and only if it is fully semiprime.
In this result, it is crucial to assume that \emph{all} ideals of the ring are either semiprime or idempotent (depending on the implication), since
in general, a semiprime ideal does not need to be idempotent, and an idempotent ideal does not have to be semiprime.

For example, let $R:=\CC[x]/x^2\CC[x]$. Then the ideal $xR\subseteq R$ is semiprime (indeed, maximal) but not idempotent, and the ideal $0\subseteq R$ is idempotent but not semiprime since $xR \not \subseteq 0 = (xR)^2$. 
%For example, let $R$ be a ring that is neither idempotent nor semiprime, and let $S$ be a ring that is idempotent and semiprime. Then the ideal $R \oplus 0$ in $R \oplus S$ is semiprime but not idempotent, and the ideal $0 \oplus S$ is idempotent but not semiprime.

Given a \ca{} $A$, Courter's result shows that \emph{all} ideals in $A$ are semiprime if and only if \emph{all} ideals in $A$ are idempotent.
Our result \autoref{prp:CharSemiprime} generalizes this to a statement about individual ideals in \ca{s}.
This is a strong generalization since many interesting \ca{s} are not fully idempotent.
\end{rmks}

%==========================================================================================
\begin{exa}
\label{exa:PedersenIdeal}
Every \ca{} contains a smallest norm-dense ideal, called its \emph{Pedersen ideal};
see \cite[Section~5.6]{Ped79CAlgsAutGp} and \cite[Section~II.5.2]{Bla06OpAlgs} for details.
By \cite[Proposition~5.6.2]{Ped79CAlgsAutGp}, the Pedersen ideal is closed under roots of positive elements.
It follows from \autoref{prp:CharSemiprime} that the Pedersen ideal is semiprime.
\end{exa}

%==========================================================================================
\begin{cor}
\label{prp:SemiprimeDixmier}
Semiprime ideals in \ca{s} are Dixmier.
In particular, prime and semiprime ideals in \ca{s} are positively spanned (hence self-adjoint and closed under scalar multiplication) and hereditary.
\end{cor}
\begin{proof}
The first claim is the last statement in \autoref{prp:CharSemiprime} that semiprime 
ideals are Dixmier.
By definition, Dixmier ideals are positive spanned and hereditary.
\end{proof}

%==========================================================================================
\begin{rmk}
While semiprime ideals in \ca{s} are always Dixmier, the converse does not hold.
Indeed, in von Neumann algebras, every ideal is Dixmier (\autoref{prp:IdealsVNA}), but in the type~$\mathrm{I}_\infty$ factor $B(\ell^2(\NN))$ not every ideal is semiprime \cite{MorSal73SemiprimeIdlsBH}.

Another example is given by the strongly invariant order ideal
\[
M := \big\{ f \in C([0,1])_+ : f(t) \leq nt \text{ for some $n \in \NN$ and all $t\in[0,1]$} \big\}
\]
in the \ca{} $C([0,1])$.
The corresponding ideal $I := \linspan(M)$ is Dixmier in $C([0,1])$, but not semiprime since it is not closed under roots of positive elements.
\end{rmk}

%==========================================================================================
The following dichotomy follows directly from \autoref{prp:CharSemiprime}.

%==========================================================================================
\begin{cor}
\label{prp:Dichotomy}
Let $I \subseteq A$ be an ideal in a \ca.
Then the sequence $(I^n)_{n \geq 1}$ of powers of $I$ is either constant (which happens precisely if $I$ is semiprime) or strictly decreasing.
\end{cor}

%==========================================================================================
Let $E$ be a right module over a \ca{} $A$.
The \emph{trace ideal} of~$E$ is the set of elements in $A$ arising as finite sums of elements in the images of module homomorphisms $E \to A$.
We refer to \cite[Section~2H]{Lam99LectModulesRgs} for details.
In particular, the trace ideal is indeed a (two-sided) ideal of~$A$ by \cite[Proposition~2.40]{Lam99LectModulesRgs}.
(While rings in \cite{Lam99LectModulesRgs} are assumed to be unital, one can check that the proof of \cite[Proposition~2.40]{Lam99LectModulesRgs} works also for non-unital rings.)

The module $E$ is said to be \emph{unital} if $E=EA$, that is, every $x \in E$ can be written as $x = x_1a_1 + \ldots + x_na_n$ for some $x_j \in E$ and $a_j \in A$.
We note that every module over a unital \ca{} is automatically unital.

The next result is of interest for studying the ideal structure of the `ring-theoretic Cuntz semigroup' of a \ca{}, which was introduced in \cite{AntAraBosPerVil23arX:CuRg, AntAraBosPerVil24arX:IdlsCuRg}.

%==========================================================================================
\begin{cor}
\label{prp:TraceIdeals}
Trace ideals of projective, unital modules over \ca{s} are semiprime.
\end{cor}
\begin{proof}
Let $P$ be a projective, unital module over a \ca{} $A$, and let~$I$ denote its trace ideal.
We show that $I$ is idempotant, following the argument in the proof of \cite[Proposition~2.40~(2)]{Lam99LectModulesRgs}.
The result then follows since idempotent ideals in \ca{s} are semiprime by \autoref{prp:CharSemiprime}.

Consider the free $A$-module $\bigoplus_P A$ over the index set $P$, which consists of tuples $(a_p)_{p \in P}$ with $a_p \in A$ for $p \in P$ such that at most finitely many $a_p$ are non-zero.
We define $\pi \colon \bigoplus_P A \to P$ by 
\[
\pi\big( (a_p)_{p \in P} \big) = \sum_{p \in P} p a_p.
\]
for $(a_p)_{p \in P} \in \bigoplus_P A$.
Then $\pi$ is a well-defined module homomorphism.
Using that~$P$ is unital, we see that $\pi$ is surjective.
Since $P$ is projective, we get a module homomorphism $\sigma \colon P \to \bigoplus_P A$ such that $\pi\circ\sigma = \mathrm{id}$.
Note that $\sigma$ is given by a family of module homomorphisms $\sigma_p \colon P \to A$ such that $\sigma(x)=(\sigma_p(x))_{p \in P}$ for every $x \in P$.

Given a module homomorphism $\varphi \colon P \to A$ and $x \in P$, we have
\[
\varphi(x)
= \varphi\big( \pi(\sigma(x)) \big)
= \varphi\Big( \sum_{p \in P} p\sigma_p(x) \Big)
= \sum_{p \in P} \varphi(p)\sigma_p(x)
\in I^2.
\]
Using that every element of $I$ is a finite sum of elements of the form $\varphi(x)$, it follows that $I$ is idempotent.
\end{proof}

%==========================================================================================
Given an ideal $I$ in a ring $R$, its \emph{semiprime closure}, denoted by $\sqrt{I}$, can be defined as the intersection of all prime ideals containing $I$, with the convention that $\sqrt{I}=R$ if $I$ is not contained in any prime ideal;
see \cite[Section~10]{Lam01FirstCourse2ed} for details. Note that $\sqrt{I}$ is 
always semiprime. 

A ring is said to be \emph{nil} if each of its elements is nilpotent (but with the degree of nilpotency possibly depending on the element).
One says that a ring $R$ is \emph{radical} over an ideal $I \subseteq R$ if for every $x \in R$ there exists $n \geq 1$ such that $x^n \in I$, that is, if $R/I$ is nil.

%==========================================================================================
\begin{thm}
\label{prp:SemiprimeClosure}
Let $I \subseteq A$ be an ideal in a \ca{}, and let $J := \interior(I)$ and $K := \exterior(I)$ denote the Dixmier interior and closure of $I$, respectively. 
Then
\begin{align*}
\sqrt{I} 
&= \sqrt{J} 
= \sqrt{K} 
= \bigcup_{n=1}^{\infty} J^{\frac{1}{n}}
= \bigcup_{n=1}^{\infty} K^{\frac{1}{n}} \\
&= \big\{ x \in A : (x^*x)^n \in I_+ \text{ for some $n \geq 1$} \big\} \\
&= \big\{ x \in A : (xx^*)^n \in I_+ \text{ for some $n \geq 1$} \big\}.
\end{align*}

In particular, $\sqrt{I}$ is the union of the increasing sequence of ideals $\big(J^{\frac{1}{n}}\big)_{n}$, all of which are radical over $I$. \end{thm}
\begin{proof}
By \autoref{prp:IntExt}, $J$ and $K$ are Dixmier ideals such that
\[
K^{1+\varepsilon} \subseteq J \subseteq I \subseteq K \subseteq J^{1-\varepsilon}
\]
for every $\varepsilon \in (0,1)$.
Since $J \subseteq K$, we have $J^{\frac{1}{n}} \subseteq K^{\frac{1}{n}}$ for each $n$, and thus $\bigcup_{n=1}^{\infty} J^{\frac{1}{n}} \subseteq \bigcup_{n=1}^{\infty} K^{\frac{1}{n}}$.
Further, $\sqrt{J}$ is by definition a semiprime ideal and therefore a Dixmier ideal satisfying $\sqrt{J} = (\sqrt{J})^{\frac{1}{2n}}$ for every $n$ by \autoref{prp:CharSemiprime}.
We also have $K^2 \subseteq J \subseteq \sqrt{J}$, and thus
\[
K^{\frac{1}{n}}
= \big( K^2 \big)^{\frac{1}{2n}}
\subseteq \big( \sqrt{J} \big)^{\frac{1}{2n}}
= \sqrt{J}.
\]
This shows that
\begin{align}
\label{prp:SemiprimeClosure:Inclusions}
\bigcup_{n=1}^{\infty} J^{\frac{1}{n}}
\subseteq \bigcup_{n=1}^{\infty} K^{\frac{1}{n}}
\subseteq \sqrt{J}
\subseteq \sqrt{I}
\subseteq \sqrt{K}.
\end{align}

Next, we show that $\bigcup_{n=1}^{\infty} J^{\frac{1}{n}}$ is a semiprime ideal.
To verify that it is closed under square roots of positive elements, let $a \in (\bigcup_{n=1}^{\infty} J^{\frac{1}{n}})_+$, and choose $n \geq 1$ such that $a \in J^{\frac{1}{n}}_+$.
Then $a^{\frac{1}{2}} \in J^{\frac{1}{2n}}_+ \subseteq \bigcup_{n=1}^{\infty} J^{\frac{1}{n}}$.
By \autoref{prp:CharSemiprime}, it follows that $\bigcup_{n=1}^{\infty} J^{\frac{1}{n}}$ is semiprime.
Since $\sqrt{K}$ is the smallest semiprime ideal in $A$ containing~$K$, and since we have $K \subseteq J^{\frac{1}{2}} \subseteq \bigcup_{n=1}^{\infty} J^{\frac{1}{n}}$, we deduce that $\sqrt{K} \subseteq \bigcup_{n=1}^{\infty} J^{\frac{1}{n}}$.
It follows that all sets in \eqref{prp:SemiprimeClosure:Inclusions} are equal.
In particular, $\sqrt{K}=\bigcup_{n \geq 1}K^{\frac{1}{n}}$, and similarly for $J$. 

By \autoref{prp:CharRootDixmierIdeal}, we have
\[
L^{\frac{1}{2n}} 
= \big\{ x \in A : (x^*x)^n \in L \big\}
= \big\{ x \in A : (xx^*)^n \in L \big\}
\]
for every Dixmier ideal $L \subseteq A$.
Applying this for $J$ and $K$, and using that $J \subseteq I \subseteq K$, we deduce the remaining equalities. 

For the last statement, the fact that $J^{\frac{1}{n}}$ is radical over $I$
follows from the fact that $(J^{\frac{1}{n}})^n = J \subseteq I$.
\end{proof}

%==========================================================================================
Given a ring $R$, the \emph{lower nilradical}, also called the \emph{prime radical}, of $R$ is defined as $\Nil_*(R) := \sqrt{0}$, that is, the semiprime closure of the zero ideal;
see \cite[Definition~10.13]{Lam01FirstCourse2ed}.
The \emph{upper nilradical} $\Nil^*(R)$ is the (unique) largest nil ideal in~$R$;
see \cite[Definition~10.26]{Lam01FirstCourse2ed}.
We have $\Nil_*(R) \subseteq\Nil^*(R)$ in every ring.

%==========================================================================================
\begin{cor}
\label{prp:NilRadicals}
Let $I \subseteq A$ be an ideal in a \ca{}, and let $J \subseteq A$ be a one-sided ideal that is radical over $I$.
Then $J \subseteq \sqrt{I}$.

In particular, we have $\Nil_*(A/I) = \Nil^*(A/I)$.
\end{cor}
\begin{proof}
%Let $J_0 \subseteq A$ be the pre-image of $J$ with respect to the quotient map $A \to A/I$.
%Then $J_0$ is a one-sided ideal in $A$ that is radical over $I$.
Assume first that $J$ is a left ideal.
Given $x \in J$, we have $x^*x \in J$, and therefore $(x^*x)^n \in I$ for some $n \geq 1$.
By \autoref{prp:SemiprimeClosure}, this implies that $x \in \sqrt{I}$, as desired.
If $J$ is a right ideal, we consider powers of $xx^*$ instead of $x^*x$.

It follows that every two-sided nil ideal in $A/I$ is contained in the prime radical of $A/I$.
\end{proof}

%==========================================================================================
\begin{rmk}
\label{rmk:KoetheConjecture}
K\"{o}the's conjecture predicts that for every ring $R$, every one-sided nil ideal in $R$ is contained in $\Nil^*(R)$;
see \cite[Conjecture~10.28]{Lam01FirstCourse2ed}.
\autoref{prp:NilRadicals} shows in particular that K\"{o}the's conjecture holds for rings that arise as the quotient of a \ca{} $A$ by a (not necessarily norm-closed) ideal $I$.

If the ideal $I$ is closed under scalar multiplication, then $A/I$ is a $\CC$-algebra, and this case is covered by Amitsur's verification \cite[Theorem~10]{Ami56AlgsInfFields} of K\"{o}the's conjecture for algebras over uncountable fields.
As noted before, there exist ideals in \ca{s} that are not closed under scalar multiplication, and for the corresponding quotient rings our verification of K\"{o}the's conjecture is new.
\end{rmk}

%==========================================================================================
\begin{thm}
\label{prp:LatticeSemiprimeIdeals}
The semiprime ideals of a \ca{} $A$ form a complete sublattice of the ideal lattice of $A$.
Thus, given an ideal $I \subseteq A$, there exists a largest semiprime ideal $I^\infty$ contained in $I$, namely:
\[
I^\infty := \bigvee \big\{ J \subseteq A \colon J \text{ semiprime ideal}, J \subseteq I \big\}.
\]
Moreover, we have
\begin{align*}
I^\infty 
= \bigcap_{n=1}^\infty I^{n}
&= \big\{ x \in A : (x^*x)^{\frac{1}{n}} \in I_+ \text{ for every $n \geq 1$} \big\} \\
&= \big\{ x \in A : (xx^*)^{\frac{1}{n}} \in I_+ \text{ for every $n \geq 1$} \big\}.
\end{align*}
\end{thm}
\begin{proof}
Since semiprime ideals can be characterized as those ideals that arise as the intersection of some family of prime ideals, it follows that the intersection of an arbitrary family of semiprime ideals is again semiprime.

To show that semiprime ideals are closed under suprema, let $(I_\lambda)_{\lambda \in \Lambda}$ be a collection of semiprime ideals.
Since each $I_\lambda$ is idempotent and Dixmier by \autoref{prp:CharSemiprime}, it follows from \autoref{prp:LatticeDixmierIdeals} that 
\[
\Bigl(\sum_{\lambda \in \Lambda}I_\lambda\Bigr)^2 = \sum_{\lambda\in\Lambda}I_\lambda^2 = \sum_{\lambda\in\Lambda}I_\lambda, 
\]
which is a semiprime ideal again by \autoref{prp:CharSemiprime}. 

\medskip

Let us verify that $I^\infty = \bigcap_{n \geq 1} I^n$.
Since $I^\infty$ is semiprime, we have $(I^\infty)^n = I^\infty$ for every $n\in\NN$ with $n \geq 1$ by \autoref{prp:CharSemiprime}.
We also have $I^\infty \subseteq I$, and therefore
\[
I^\infty = (I^\infty)^n \subseteq I^n.
\]
Since this holds for every $n \geq 1$, we get $I^\infty \subseteq \bigcap_{n \geq 1} I^n$.

Set $J := \interior(I)$ and $K := \exterior(I)$.
By \autoref{prp:IntExt}, we have
\[
\ldots \subseteq J^2 \subseteq I^2 \subseteq K^2 \subseteq J \subseteq I \subseteq K.
\]
It follows that $\bigcap_{n \geq 1} I^n = \bigcap_{n \geq 1} K^n = \bigcap_{n \geq 1} J^n$, and in particular $\bigcap_{n \geq 1} I^n$ is a Dixmier ideal.
By \autoref{prp:LatticeDixmierIdeals}, we see 
\[
\Bigl( \bigcap_{n \geq 1} K^n \Bigr)^2 
= \bigcap_{n \geq 1} K^{2n} 
= \bigcap_{n \geq 1} K^{n}, 
\]
which implies $\bigcap_{n \geq 1} K^{n} = \bigcap_{n \geq 1} I^{n}$ is idempotent and thus semiprime by \autoref{prp:CharSemiprime}. 
Since it is contained in $I$, and $I^\infty$ is the largest semiprime ideal contained in $I$, we conclude that $\bigcap_{n \geq 1} I^n \subseteq I^\infty$.

Finally, using \autoref{prp:CharRootDixmierIdeal}, we observe 
\[
\bigcap_{n \geq 1} I^{n} 
= \bigcap_{n \geq 1} J^{\frac{n}{2}} 
\subseteq \big\{ a\in A : (a^*a)^{\frac{1}{n}}\in I_+ \text{ for every $n$} \big\}
\subseteq \bigcap_{n \geq 1} K^{\frac{n}{2}} 
= \bigcap_{n \geq 1} I^{n} , 
\]
and all of them must be equal. 
The remaining expression of $I^\infty$ can be verified similarly. 
\end{proof}

%==========================================================================================
\begin{rmks}
\label{rmk:LatticeSemiprimeIdeals}
(1)
The norm-closed ideals in a \ca{} $A$ form a complete lattice, but in general it is not a complete sublattice of the lattice of ideals in~$A$.
Indeed, given a family $(I_\lambda)_j$ of norm-closed ideals, the ideal $\sum_{\lambda\in\Lambda} I_\lambda$ is not necessarily norm-closed.
The supremum of this family in the lattice of norm-closed ideals is the closure $\overline{\sum_{\lambda\in\Lambda} I_\lambda}$.
Consequently, an ideal in a \ca{} typically does not contain a largest norm-closed ideal. For example, consider $c_c(\NN)\subseteq c_0(\NN)$.
From that perspective, semiprime ideals form a better behaved class than the norm-closed ideals.

(2)
\autoref{prp:LatticeSemiprimeIdeals} shows in particular that the union of an ascending chain of semiprime ideals in a \ca{} is again semiprime.
This property is known to hold in commutative rings, but fails in general noncommutative rings; 
see \cite{GreRowVis16ChainsPrime}.

It would be interesting to determine if the union of an ascending chain of prime ideals in a \ca{} is again prime.
\end{rmks}

%==========================================================================================
If $I \subseteq A$ is a norm-closed ideal in a \ca, and $B \subseteq A$ is a \csuba{}, then $I \cap B$ is a norm-closed ideal in $B$.
We show that an analogous result holds for semiprime ideals.

%==========================================================================================
\begin{prp}
\label{prp:IntertsectionSemiprimeSubalgebra}
Let $I \subseteq A$ be a semiprime ideal in a \ca, and let $B \subseteq A$ be a \csuba{}.
Then $I \cap B$ is a semiprime ideal in $B$.
\end{prp}
\begin{proof}
By \autoref{prp:CharSemiprime}, it suffices to show that $I \cap B$ is closed under roots of positive elements.
This follows easily, since both $I$ and $B$ have this property.
\end{proof}

%==========================================================================================
\begin{rmk}
The analog of \autoref{prp:IntertsectionSemiprimeSubalgebra} for prime ideals does not hold.
For example, take any unital, simple \ca{} $C$ with a \csuba{} $D \subseteq C$ with $D \cong C([0,1])$.
Then consider the direct sum $A := C \oplus C$, the \csuba{} $B := D \oplus D$, and the ideal $I := C \oplus 0$.
Then $I$ is maximal and therefore prime.
But $I \cap B$ is the ideal $C([0,1])\oplus 0$ inside $C([0,1])\oplus C([0,1])$, which is not prime.
\end{rmk}

%==========================================================================================
It is well-known that every maximal ideal in a \emph{unital} \ca{} is norm-closed.
Surprisingly, it is not known if this also holds in non-unital \ca{s}.
This was asked by Narutaka Ozawa (see also \cite[Question~2.8]{GarThi24PrimeIdealsCAlg}).
We note that maximal ideals in \ca{s} are prime and therefore enjoy the properties obtained in \autoref{prp:SemiprimeDixmier}, which are well-known to hold for norm-closed ideals.

%==========================================================================================
\begin{prp}
\label{prp:MaximalIdeal}
Maximal ideals in \ca{s} are prime. 
In particular, they are always positively spanned (hence self-adjoint and closed under scalar multiplication) and hereditary.
\end{prp}
\begin{proof}
It is a folklore result that maximal ideals in idempotent rings are prime.
We include the details for the convenience of the reader.
Let $I \subseteq A$ be a maximal ideal in a \ca.
To show that $I$ is prime, let $J,K \subseteq A$ be ideals such that $JK \subseteq I$.

To reach a contradiction, assume that neither $J$ nor $K$ are contained in $I$.
Since~$I$ is a maximal ideal, this implies that $I+J=A$ and $I+K=A$.
But then
\[
A = A^2 = (I+J)(I+K)
\subseteq I,
\]
a contradiction.
Now the result follows from \autoref{prp:SemiprimeDixmier}.
\end{proof}

%==========================================================================================
%==========================================================================================
%\bibliographystyle{../../aomalphaMyShort}
%\bibliography{../../References}

\providecommand{\etalchar}[1]{$^{#1}$}
\providecommand{\bysame}{\leavevmode\hbox to3em{\hrulefill}\thinspace}

\end{document}